\documentclass[a4paper,11pt]{article}

\usepackage{tabularx}
\usepackage{graphicx}

\usepackage{amssymb,latexsym,amsmath,amsthm}

\usepackage{makeidx}

\usepackage{multicol} 

\usepackage[all]{xy}

\usepackage{pst-node} 
\usepackage{pstricks}
\usepackage{pstricks-add}

\usepackage[a4paper,left=3.4cm,right=3.4cm,top=\dimexpr17mm+4.3\baselineskip,bottom=3.8cm]{geometry} 

\usepackage{fancyhdr}
\usepackage[colorlinks, linkcolor=black, pagebackref, urlcolor=blue, citecolor=blue]{hyperref}

\usepackage{verbatim} 

\usepackage{manfnt}
\usepackage{mathrsfs}
\usepackage{youngtab}

\swapnumbers 
\theoremstyle{plain}
\newtheorem{thm}{Theorem}[section]
\newtheorem{lemma}[thm]{Lemma}
\newtheorem{prop}[thm]{Proposition}
\newtheorem{cro}[thm]{Corollary}
\newtheorem{conj}[thm]{Conjecture}

\theoremstyle{definition}
\newtheorem{eg}[thm]{Example}
\newtheorem{dfn}[thm]{\textbf{Definition}}
\newtheorem{exer}{Exercise}
\newtheorem{rem}[thm]{Remark}
\newtheorem{obs}[thm]{Observation}
\newtheorem{des}[thm]{}
\newtheorem{ques}[thm]{Question}
\newtheorem{app}[thm]{Application}
\newtheorem{his}[thm]{Historical Note}
\newtheorem{notation}[thm]{Notation}
\newtheorem{conv}[thm]{Convention}

\newtheorem{cons}[thm]{Construction}
\newtheorem{aim}[thm]{Aim}
\newtheorem*{mconj}{Conjecture}
\newtheorem*{mthm}{Theorem}

\newcommand{\bthm}{\begin{thm}}
\newcommand{\ethm}{\end{thm}}

\newcommand{\bprop}{\begin{prop}}
\newcommand{\eprop}{\end{prop}}

\newcommand{\bcro}{\begin{cro}}
\newcommand{\ecro}{\end{cro}}

\newcommand{\bex}{\begin{ex}}
\newcommand{\eex}{ \end{ex}}

\newcommand{\bdf}{\begin{dfn}}
\newcommand{\edf}{ \end{dfn}}

\newcommand{\brem}{\begin{rem}}
\newcommand{\erem}{ \end{rem}}

\newcommand{\blem}{\begin{lemma}}
\newcommand{\elem}{\end{lemma}}

\newcommand{\bhis}{\begin{his}}
\newcommand{\ehis}{ \end{his}}

\newcommand{\beg}{\begin{eg}}
	\newcommand{\eeg}{ \end{eg}}

\newcommand{\bobs}{\begin{obs}}
\newcommand{\eobs}{ \end{obs}}

\newcommand{\bques}{\begin{ques}}
\newcommand{\eques}{\end{ques}}

\newcommand{\bdes}{\begin{des}}
\newcommand{\edes}{\end{des}}

\newcommand{\bproof}{\begin{proof}}
\newcommand{\eproof}{\end{proof}}

\newcommand{\bexr}{\begin{exer}}
\newcommand{\eexr}{ \end{exer}}

\newcommand{\bnote}{\begin{notation}}
\newcommand{\enote}{\end{notation}}

\newcommand{\bconj}{\begin{conj}}
\newcommand{\econj}{ \end{conj}}

\newcommand{\bconv}{\begin{conv}}
\newcommand{\econv}{ \end{conv}}

\newcommand{\bapp}{\begin{app}}
\newcommand{\eapp}{ \end{app}}

\newcommand{\bcons}{\begin{cons}}
\newcommand{\econs}{ \end{cons}}

\newcommand{\baim}{\begin{aim}}
\newcommand{\eaim}{\end{aim}}

\newcommand{\bmthm}{\begin{mthm}}	\newcommand{\emthm}{\end{mthm}}

\newcommand{\bmconj}{\begin{mconj}}	\newcommand{\emconj}{\end{mconj}}

\newcommand{\bnum}{\begin{enumerate}}

\newcommand{\enum}{\end{enumerate}}

\newcommand{\barray}{\begin{array}{l}}

\newcommand{\earray}{\end{array}}

\newcommand{\bmeq}{\begin{eqnarray*}}

\newcommand{\emeq}{\end{eqnarray*}}

\newcommand{\bmeqnum}{\begin{eqnarray}}

\newcommand{\emeqnum}{\end{eqnarray}}

\newcommand{\Ext}{\operatorname{Ext}}
\newcommand{\Gl}{\operatorname{GL}}

\newcommand{\X}{\operatorname{\mathbb{X}}}

\newcommand{\h}{\operatorname{H}}
\newcommand{\R}{\operatorname{R}}
\newcommand{\tor}{\operatorname{Tor}}
\newcommand{\Hom}{\operatorname{Hom}}

\newcommand{\End}{\operatorname{End}}

\newcommand{\coker}{\operatorname{Coker}}
\newcommand{\node}[1]
{\underset{#1}{\bullet}}
\newcommand{\nod}[1]{\underset{#1}{\circ}}
\newcommand{\kerr}{\textsf{Ker}}

\newcommand{\mmod}{\textsf{-}\mathsf{mod}}
\newcommand{\mMod}{\textsf{-}\mathsf{Mod}}

\newcommand{\mproj}{\textsf{-}\mathsf{proj}}

\newcommand{\stab}{\textsf{-}\mathsf{stab}}

\newcommand{\perf}{\textsf{-}\mathsf{perf}}

\newcommand{\simto}{\,\mathop{\longrightarrow}\limits^\sim\,}

\newcommand{\bbG}{\mathbb{G}}
\newcommand{\bbT}{\mathbb{T}}
\newcommand{\bbP}{\mathbb{P}}
\newcommand{\bbU}{\mathbb{U}}
\newcommand{\bbB}{\mathbb{B}}
\newcommand{\bbL}{\mathbb{L}}

\newcommand{\xnd}{\operatorname{\mathbb{X}_{n,d}}}

\newcommand{\rgc}{\R \Gamma_c}

\newcommand{\dli}{\R_{\bbL}^\bbG}
\newcommand{\dlr}{{}^*\R_{\bbL}^\bbG}
\newcommand{\sr}{^* \! \mathscr{R}^{\Gl_{n}}_{\Gl_{n-1}}}

\makeindex

\AtEndDocument{\bigskip{\footnotesize%
		\textsc{Université de Paris, UFR de Mathématiques, Bâtiment Sophie Germain, 5 rue Thomas Mann, 75205 Paris CEDEX 13, France} \par  
		\textbf{E-mail address:}  \texttt{parisa.ghazizadeh@imj-prg.fr}
}}

\begin{document}

		\title{Mod $\ell$ cohomology of some Deligne--Lusztig varieties for $\Gl_n(q)$}
		\author{Parisa Ghazizadeh}
		\date{}
		\maketitle

		\begin{abstract}
			In this article, we study the mod $\ell$ cohomology of some Deligne--Lusztig varieties for $\Gl_n(q)$. We prove that the cohomology groups of these varieties are torsion-free under some conditions on the characteristic.  Under the torsion-free assumption we can compute the cohomology groups explicitly and we prove that the cohomology complex satisfies a partial-tilting condition, which is one of the necessary conditions in the geometric version of  Brou\'{e}'s abelian defect group conjecture. 
		\end{abstract}
	
\tableofcontents 
\section{Introduction}
Deligne--Lusztig varieties were originally introduced (in \cite{DL}) by 
Deligne and Lusztig to understand the ordinary representation theory (that is, over fields of characteristic zero) of finite reductive groups. 
In the ordinary setting, Lusztig (in \cite{Lus-character}) 
gave a complete classification of the irreducible characters of finite reductive groups using cohomology groups of Deligne--Lusztig varieties.
In the modular setting (that is, over fields positive characteristic), much less is known for the representations of finite reductive group $G$, even though
Deligne--Lusztig’s construction can be adapted to 
the modular setting (see for example the work of  Brou\'e \cite{Bro-isomet} 
and Bonnaf\'e--Rouquier \cite{Rouq-bonnDL}). 
One can replace the  cohomology groups of a Deligne--Lusztig variety with the cohomology complex, which encodes more information than the individual cohomology groups in the modular setting. This point of view was first suggested by Brou\'e as a strategy towards his abelian defect group conjecture, stated in 1988.

\begin{mconj}[{Abelian defect group conjecture, \cite{conj_broue}}] 
	Let $G$ be a finite group. Let $B$ be a block of $\overline{\mathbb{Z}}_\ell G$ of defect $D$ and let $b$ be its  correspondent of $
	\overline{\mathbb{Z}}_\ell N_{G}(D)$.
	If $D$ is abelian, then $$D^{b}(b \mmod) \simeq D^{b}(B \mmod).$$
\end{mconj}	
Brou\'e  predicted that the cohomology complex $\R \Gamma_c(\X,\overline{\mathbb{Z}}_\ell )$ of a suitably chosen Deligne--Lusztig variety $\X$ should induce the derived equivalence. 
He stated a version of his conjecture for finite reductive groups and unipotent blocks, known as the ``geometric version'' of the conjecture. The varieties involved in the geometric version are a parabolic version of the classical Deligne--Lusztig varieties.
\smallskip

In order to induce a derived equivalence, the cohomology complex $D$ of the parabolic Deligne--Lusztig variety has to be partial-tilting, that is, to be a perfect complex such that all maps from $D$ to the shifted complex $D[n]$ are null-homotopic for any $n\neq 0$. Verifying this condition is one of the main difficulties in proving this conjecture.
\smallskip

In \cite{Malle-Michel-Broue}, Brou\'e, Malle and Michel made some progress on the geometric version of the abelian defect group conjecture for \emph{large} prime numbers $\ell$.
When $\ell$  is a prime number, larger than the \emph{Coxeter number}, the properties of $\ell$-blocks of $G$ depend only on the cyclotomic polynomial $\Phi_d(q)$ that $\ell$ divides (see \cite{l-sylow-conjugate}). This observation motivated the development of a theory of \emph{$\Phi_d$-blocks} of finite reductive groups, which are the generic blocks for any prime number $\ell$ dividing $\Phi_d(q)$  by Brou\'e, Malle and Michel in \cite{Malle-Michel-Broue}. 
\smallskip

If we assume that $\ell$ does not divide $q$ and $\ell \mid |G|$ then $\ell \mid \Phi_d(q)$ for some $d$. In addition if we assume that $\ell$ is larger than the Coxeter number (the largest integer $h$ such that $\Phi_{h}(q) \mid |G|$) then $d$ is unique and it is the order of $q$ in $\mathbb{F}_{\ell}^{\times}$. In other words, it is the smallest positive integer $d$ with $q^d \equiv 1 (\mathrm{mod} \ \ell)$. 
\smallskip

In this paper,  we assume that $\ell$ is a prime number which is large that divides the order of the group and $\ell \nmid q$. We focus on the finite reductive group $\Gl_n(q)$  We study the cohomology groups and cohomology complex of specific Deligne--Lusztig varieties, which are those varieties involved in the geometric version of Brou\'e's conjecture for unipotent $\Phi_d$-blocks of $\Gl_n(q)$. These varieties, which are denoted by $\X_{n,d}$ were explicitly defined in \cite{parabolicDM} and \cite{cohomology-of-DL_Dudas}.
 In the ordinary setting, the cohomology groups of these varieties have been explicitly described by Lusztig for $d = n$ in \cite{Lusztig-coxeter}, by Digne--Michel--Rouquier for the case $d = n-1$ in \cite{DMR} and in \cite{cohomology-of-DL_Dudas} by Dudas for general $d$. In this case, the unipotent representations of $\Gl_n(q)$  are labeled by the irreducible representations of its Weyl group $\mathfrak{S}_n$ which are in turn labeled by the partitions of $n$.
In \cite{cohomology-of-DL_Dudas} explicit formulas are given for $\h^{\bullet}_c(\X_{n,d},\mathbb{Q}_\ell)$ in terms of this parametrization. 
\smallskip

Our first aim is the generalization of that description in the modular setting. The first step towards this direction, and one our main results, is the following theorem (see \ref{torsion-free}):
\begin{mthm}
	If  $\ell \mid \Phi_m(q)$, where $m>d $, and  $m>n-d+1$, and $m > 6$, then $\h_c^{\bullet}(\X_{n,d},\mathbb{Z}_\ell)$ is torsion-free over $\mathbb{Z_\ell}$.
\end{mthm}	
The  assumption on $m$ ensures that the principal $\ell$-block is a block with cyclic defect group; it also guarantees that the cohomology complex of $\X_{n,d}$ is a perfect complex.  The torsion-free result allows us to explicitly calculate each cohomology group $\h^i_c(\X_{n,d},\mathbb{Z}_\ell)$, along with the action of the Frobenius map. Given a partition $\lambda$ of $n$, define a $\mathbb{Z}_\ell \Gl_n(q)$-lattice $\nabla_{\mathbb{Z}_\ell}(\lambda)$ whose character is the unipotent character corresponding to $\lambda$ with the notation of \ref{def:Xnd}. We show that:
\begin{mthm} 
	 Assume $\ell \mid \Phi_m(q)$ with $m>d$ and $m>n-d+1$. If $\h^{\bullet}(X_{n,d},\mathbb{Z}_\ell)$ 
	is torsion-free over $\mathbb{Z}_\ell$ then
	the cohomology of $\X_{n,d}$ over $\mathbb{Z}_\ell$ is given by the following formulas:
	\begin{equation}
	\h^{x}(\X_{n,d},\mathbb{Z}_\ell)=\begin{cases}
	\nabla_{\mathbb{Z}_\ell}(\mu * x) \langle q^{x} \rangle, &  0\leq x<n-d \text{ and } x\leq d-1, \\
	\nabla_{\mathbb{Z}_\ell}(\mu * (x+1)) \langle q^{x+1}\rangle  &  n-d\leq x < d-1, \\
	\nabla_{\mathbb{Z}_\ell}(n) \langle q^n \rangle & x=2n-d-1, \\
	0 & \text{otherwise}.
	\end{cases}
	\end{equation}
\end{mthm}	
Under the same assumptions on $\ell$, we can find a representative of the cohomology complex of $\X_{n,d}$ as a bounded complex of finitely generated projective modules. The explicit description of this complex allows us to show that $\R\Gamma_c(\X_{n,d},\mathbb{F}_\ell)$ is a partial-tilting complex, hence satisfies  the self-orthogonality condition (see \ref{result-partial_tilting}):
\begin{mthm}
		If  $\ell \mid \Phi_m(q)$, where $m>d$, and  $m>n-d+1$, and $m>6$, then for all non-zero integers $a$, $$\Hom_{D^b(\mathbb{F}_\ell \Gl_n(q)\mmod)}(\R \Gamma_c \big(\X_{n,d},\mathbb{F}_\ell),\R \Gamma_c(\X_{n,d},\mathbb{F}_\ell)[a] \big)=0.$$
\end{mthm}

\section{Preliminaries}
\subsection{Some homological algebra}
Given $A$ a ring with unit, we denote by $A\mmod$ the category of finitely generated left $A$-modules and by $A\mMod$  the category of left $A$-modules. If $\mathcal{A}= A\mMod$ or $A \mmod$, we denote  by $K(\mathcal{A})$ the corresponding homotopy category and by $D(\mathcal{A})$ the corresponding derived category. We write $K^b(\mathcal{A})$ and $D^b(\mathcal{A})$ respectively for the full subcategories of $K(\mathcal{A})$ and $D(\mathcal{A})$ whose objects are bounded. We can define the usual derived bifunctors $\R \Hom_{A}(-,-)$ and $- \bigotimes_{A} -$
since $A \mmod$ and $A \mMod$ have enough projective objects. 

Recall that the stable category of finitely generated $A$-modules  which is denoted by $A\stab$ is defined as the category such that:
	\begin{itemize}
		\item The objects are finitely generated $A$-modules.
		\item For arbitrary objects $X$ and $Y$
		\begin{center}
			$\Hom_{A\stab}(X,Y)=\Hom_{A\mmod}(X,Y)/\approx$
		\end{center}	
		\noindent
		where $f \approx g$ if and only if
		the morphism $f-g$ factors through a projective module.
	\end{itemize}	
By definition, any projective module is a zero object in the stable category. We can endow $A\stab $ with  the structure of triangulated category. The suspension is given by the Heller operator which we recall the construction now. If $A$ is a finite dimensional $A$-algebra over field $R$ then any finitely generated $A$-module has a finitely generated projective cover ($A$ is said to be semiperfect). In particular, if $A$ is a symmetric algebra consequently any finitely generated $A$-module admits an injective resolution.
\smallskip

Let $M$ be a finitely generated $A$-module. The Heller operator
	 which is denoted by $\Omega$ is defined by
	\begin{center}
		$\Omega M:= \kerr(P_{M} \twoheadrightarrow M)$
	\end{center}	
	where $P_{M}$ is a projective cover of $M$. Then, one can define inductively $\Omega^{n}M:=\Omega(\Omega^{n-1}M)$ for any $n>1$. By convention $\Omega^{0}M$ is defined as the minimal submodule of $M$ such that $M / \Omega^{0}M$ is a projective module.
\begin{prop}[{\cite[Proposition 1.23]{OLDL}}] \label{stabb} 
	Let $M$ and $N$ be finitely generated  $A$-modules. 
	\begin{itemize}
		\item 	For every $n \geq1$
		$$
		\Hom_{A\stab}(\Omega^n M,N) \simeq \Ext_{A}^n(M,N).
		$$
		\item Let $M$ and $N$ be two indecomposable non-projective modules. If $M \simeq N$ in $A\stab$ then $M \simeq N$ in $A\mmod$.
	\end{itemize}
\end{prop}
Let $M$ be a finitely generated $A$-module. We can also define $\Omega^{-1}M:=\coker(M \hookrightarrow I_{M})$ where $I_{M}$ is the injective hull of $M$. 
\begin{prop}
	Let  $0 \longrightarrow X \longrightarrow Y \longrightarrow Z \longrightarrow 0  $ be an exact sequence in $A\mmod$. This yields an exact sequence 
	\[
	\xymatrix{ X  \ar[r] & Y  \ar[r] & Z \ar[r] & \Omega^{-1} X &}
	\]
	in $A\stab$.  If $\Omega^{-1}$ is considered as the shift functor in a triangle, it gives us the structure of a triangulated category on $A\stab$.  
\end{prop}	
Recall that a perfect complex of $A$-modules is a complex which is quasi-isomorphic to a bounded complex of finitely generated projective modules. 

Using \cite{Rickard-stab}, 
the functor  $A\mmod \longrightarrow D^{b}(A\mmod)$ induces an equivalence of triangulated categories
\begin{equation}\label{perfect_complex}
A\stab \simto  D^{b}(A\mmod) /A\perf.
\end{equation}	
where $A\perf$ is the category of perfect complexes of $A$-modules.
\begin{lemma} \label{stabomega} \label{omegarick}
	Let $C$ be a perfect complex such that $\h^{i}(C)=0$ for $i \neq s,t$ where $s>t$. Then $\h^{t}(C) \simeq \Omega^{s-t+1}\h^{s}(C)$ in $A\stab$.
\end{lemma}	
\begin{proof}
	We have a distinguished triangle  in $D^{b}(A\mmod)$
	\[
	\xymatrix{ \h^{t}(C)[-t] \ar[r] &  C \ar[r] & \h^{s}(C)[-s] \ar@{~>}[r] & . }
	\] 
	We can shift it to obtain the following triangle
	\[
	\xymatrix{ C[s] \ar[r] &  \h^{s}(C)[0] \ar[r] & \h^{t}(C)[s-t+1] \ar@{~>}[r] & . }
	\] 
	We conclude using the image of this triangle in $D^b(A\mmod)/A\perf$.
\end{proof}	
 The following theorem, stated by Rickard, gives conditions for a complex to induce a derived equivalence.
\begin{thm}[{Rickard, \cite[Chapter 1]{Keller-der2}}] \label{rick1}
	Let $K$ be a field. Let $A$ and $B$ be finite dimensional $K$-algebras. 
	The following conditions are equivalent:
	\begin{enumerate}
		\item  $ D^b(A\mmod) \simeq D^b(B\mmod)$.
		\item There exists a complex $T$ of $B$-modules such that the following conditions hold
		\begin{itemize}
			\item $T$ is perfect,
			\item $T$ generates $D^b(B\mmod)$ as a triangulated category with infinite direct sums,
			\item $\Hom_{D^{b}(B\mmod)}(T, T[n]) = 0$ for every non-zero integer $n $, 
			\item  $\End_{D^{b}(B\mmod)}(T)\simeq A$ as $K$-algebra.
		\end{itemize}
	\end{enumerate}
\end{thm}
In the case where $B$ is a unipotent block of a finite reductive group, Brou\'{e} conjectured that $T$ can be chosen to be the cohomology complex of a suitably chosen Deligne--Lusztig variety.
\subsection{Finite reductive groups and Deligne--Lusztig varieties}
Let  $\ell \neq p$ be a prime number. Throughout this paper, we will work with an $\ell$-modular system $(K,\mathcal{O},k)$ such that $K$ is a finite extension of the field of $\ell$-adic numbers $\mathbb{Q}_\ell$, $\mathcal{O}$ is the integral closure of the ring of $\ell$-adic integers in $K$ and $k$ is the residue field of the local ring $\mathcal{O}$. We also assume that $K$ is large enough, that is, it contains all primitive $|G|$-th roots of unity. 
\smallskip

	Let $\mathbb{G}$ be a connected algebraic group over $\overline{\mathbb{F}}_p$, and   $F:\mathbb{G} \longrightarrow \mathbb{G}$ be a Frobenius endomorphism. Then $G:=\mathbb{G}^F$ (the set of fixed points of $\mathbb{G}$)  is called a finite reductive group or a finite group of Lie type.
\smallskip

Let $\bbT\in \bbB$ be a maximal torus of $\bbG$ contained in a  Borel subgroup, $W :=N_{\bbG}(\bbT) /\bbT$ and $S$ be a set of simple reflections associated to $\bbB$.
 Let $I$  be a subset of $S$ and  $\mathbb{P}_{I}$ be a standard parabolic subgroup of $\mathbb{G}$ containing  $\bbU$, and $\mathbb{U}_{I}$ be its unipotent radical. Let $\mathbb{L}_I$ be the standard Levi subgroup of $\mathbb{P}_I$ containing $\bbT$.
Let  $W_{I}$ be the  parabolic subgroup of $W$ generated by $I$ which is the Weyl group of $\bbL_I$.  Assume that $w$ is $I$-reduced (which means that $w$ has minimal length in the coset $W_I w$) and that $^{w}I = I$ and $\dot{w}$ is the  representative of $w \in W$ in $N_{\bbG}(\bbT)$. Then the parabolic Deligne--Lusztig varieties associated with the pair $(I, w)$ are defined by
\begin{align*}
&&	\tilde{\X}(I,\dot{w}F):&= \{g\mathbb{U}_I  \in \mathbb{G}/\mathbb{U}_I  \mid g^{-1}F(g) \in \mathbb{U}_I \dot{w}^F \mathbb{U}_{I}\}  &&\\
\text{and} &&  \X(I,wF):&=\{g \bbP_I  \in \mathbb{G} /\bbP_I  \mid g^{-1}F(g) \in \mathbb{P}_I w ^F \mathbb{P}_{I}\}.&&
\end{align*}
The finite group $G=\mathbb{G}^F$ acts on $\X(I,wF)$ and $\tilde{\X}(I,wF)$ by left multiplication and $\mathbb{L}_I^{\dot{w}F}$ acts on $\tilde{\X}(I,wF)$ by right multiplication. 
 The varieties $\X(I,wF)$ and $\tilde{\X}(I,\dot{w}F)$ are quasi-projective varieties and their dimension is $\ell(w) := \dim \mathbb{U}_I - \dim(\mathbb{U}_I \cap ^{wF}\mathbb{U}_I ) $  (see \cite[page 22]{parabolicDM}) and the map $\mathbb{G}/ \mathbb{U}_I \longrightarrow \mathbb{G}/ \mathbb{P}_I $ induces a $G$-equivariant isomorphism: 
\begin{equation} \label{eq:DLquo}
\tilde{\X}(I,\dot{w}F)/\mathbb{L}_I^{\dot{w}F} \simeq \X(I,wF).
\end{equation}	

\subsection{Cohomology of Deligne--Lusztig varieties}
\label{dlind.}
Let $R$ be any ring among the modular system $(K,\mathcal{O},k)$. The cohomology complexes of Deligne--Lusztig varieties $\X(I,wF)$ and $\tilde{\X}(I,\dot{w}F)$ with coefficients in $R$ are denoted by $\R\Gamma_c(\X(I,wF),R)$ and $\R\Gamma_c(\tilde{\X}(I,\dot{w}F),R)$ respectively. Recall from \cite{Deligne-etale} that $\R\Gamma_c(\X(I,wF),R)$ (resp. $\R\Gamma_c(\tilde{\X}(I,\dot{w}F),R)$) is a bounded complex of finitely generated $RG$-modules  (resp. $(RG,R\mathbb{L}^{\dot{w}F}_I)$-bimodules). Since $\bbL^{wF}_I$ acts freely on  $\tilde{\X}(I,\dot{w}F)$ (similar to \cite[Proposition 12.1.10]{Sl2}), by \cite[Proposition 6.4]{DL} and \eqref{eq:DLquo} we have an isomorphism
\begin{align} \label{cohomologyX}
\R\Gamma_{c}(\X(I,wF),R)\simeq \R\Gamma_c(\tilde{\X}(I,\dot{w}F),R)\otimes^L_{R\mathbb{L}_I^{\dot{w}F}} R.
\end{align}
in $D^b(\mathbb{L}_I^{\dot{w}F}\mmod)$.
\smallskip	
We define the triangulated functors between $D^{b}(R \bbL_I^{\dot{w}F} \mmod)$ and $D^{b}(R G \mmod)$:
\begin{align*}
&&\mathscr{R}_{\bbL_I}^{\bbG,w}(-)&:=\rgc(\tilde{\X}(I, \dot{w}F),R) \otimes^{L}_{ R \bbL_I^{\dot{w}F}}-
&&\\
\text{and} &&^* \! \mathscr{R}_{\bbL_I}^{\bbG,w}(-)&:=\R\Hom_{R G}(\rgc(\tilde{\X}(I, \dot{w}F),R),-). &&
\end{align*}	
These functors are called Deligne--Lusztig induction and restriction functors respectively. In the case $w=\dot{w}=1$, these become the Harish-Chandra induction and restriction respectively.
\smallskip

According to \cite[Proposition 3.4.19]{groth_HD}, these functors induce morphisms between the corresponding Grothendieck groups
$$\begin{aligned} \label{grotgroup1}
\R^{\bbG,w}_{\bbL_I}(-) : & \ K_0(R \bbL_I^{\dot{w}F}\mmod) \longrightarrow K_0(RG\mmod), \\
{}^*\R^{\bbG,w}_{\bbL_I}(-) : &\ K_0(RG\mmod) \longrightarrow K_0(R \bbL_I^{\dot{w}F}\mmod).
\end{aligned}$$
For simplicity of our notation, we  remove $w$ form the notations of the above functors and maps.
\begin{prop}\label{boun1}
	Assume that $\ell \nmid |\mathbb{L}_I^{\dot{w}F}|$. Then
	\begin{enumerate}
		\item  The bounded complex $\R\Gamma_c(\X(I,wF),R)$  is a perfect complex of $RG$-modules.
		\item For all $i< \ell(w) $,  $\h_c^i(\X(I,wF),R) = 0$.
		\item $\h_c^{\ell(w)}(\X(I,wF), \mathbb{Z}_\ell)$ is torsion-free.
	\end{enumerate}
\end{prop}
\begin{proof}
	Let $x\bbU_I$ be an element of $\tilde{\X}(I,\dot{w}F)$ then 
	$$ \mathrm{Stab}_{G}(x\bbU_I)=\{g \in G \mid gx\bbU_I=x\bbU_I \} \subset  (x\bbU x^{-1})^F. $$
	Since $\bbU_I$ is a unipotent group then $(x\bbU_Ix^{-1})^F$ is $p$-group (see \cite[Section 2.1]{Malle-Test}). With the assumption $\ell \neq p$, the order of $\mathrm{Stab}_{G}(x)$ is invertible in $R$. Consequently it follows from \cite[Corollary 2.3]{OLDL} that $\R\Gamma_c(\tilde{\X}(I,\dot{w}F),R)$  is a perfect complex of $RG$-modules.
	
	By \ref{cohomologyX}, if $\ell \nmid |\bbL^{\dot{w}F}|$ then $\R \Gamma_{c}(\X(I,wF),R)$ is a direct summand of the complex  $\R \Gamma_c(\tilde{\X}(I,\dot{w}F),R)$  and therefore $\R \Gamma_{c}(\X(I,wF),R)$ is also a perfect complex in that case which proves $(i)$.
	\smallskip
	
	For (ii), we relate the parabolic Deligne--Lusztig varieties to the non-parabolic case. Let $I\subseteq J$ be two subsets of simple reflections. By \cite[Proposition 8.22]{parabolicDM},
	for the elements $w \in W$ and $v \in W_{I}$, we have an isomorphism of $\mathbb{G}^F$-varieties-$\mathbb{L}_{J}^{\dot v \dot{w}F}$ 
	\begin{center}
		$\tilde{\X}_{\bbG}(I,wF) \times_{\mathbb{L}_{I}^{wF}} \tilde{\X}_{\mathbb{L}_I}(J,v(wF)) \simeq \tilde{ \X}_{\mathbb{G}}(J,vwF)$.
	\end{center}	
	We note that for the variety $\tilde{\X}_{\mathbb{G}}(I,wF)$, we have considered the action of Frobenius map $F$ and for $\tilde{\X}_{\mathbb{L}_{I}}(J,vF) $ and we have considered the action of $wF$.
	Moding out by the finite group $\bbL_J^{\dot v \dot wF}$ we obtain
	\begin{center}
		$\tilde{\X}_{\mathbb{G}}(I,wF) \times_{\mathbb{L}_{I}^{wF}} \X_{\mathbb{L}_I}(J,v(wF)) \simeq  \X_{\mathbb{G}}(J,vwF)$.
	\end{center}	
	Let us assume $J=\emptyset$ and $v=1$. Then we have
	\begin{center}
		$\tilde{\X}_{\mathbb{G}}(I,wF) \times_{\mathbb{L}_{I}^{wF}} \X_{\mathbb{L}_I}(wF) \simeq  \tilde{\X}_{\mathbb{G}}(I,wF) \times_{\mathbb{L}_{I}^{wF}} \mathbb{L}_{I}^{wF}/ (\mathbb{B}\cap \mathbb{L}_{I})^{wF} \simeq   \X_{\mathbb{G}}(wF)$.
	\end{center}	
	Now, if we look at the cohomology of these varieties with the assumptions $\ell \nmid |\mathbb{L}_{I}^{wF}|$ then
	\begin{center}
		$\R\Gamma_c (\tilde{\X}_{\mathbb{G}}(I,wF),R)  \otimes_{R \mathbb{L}_{I}^{wF}}  R[ \mathbb{\mathbb{L}}_{I}^{wF}/  (\mathbb{B}\cap \mathbb{L}_l)^{wF}] \simeq  \R\Gamma_c(\X_{\mathbb{G}}(wF),R)$.
	\end{center}	
	On the other side,
	\begin{center}
		$R [\mathbb{\mathbb{L}}_{I}^{wF}/  (\mathbb{B}\cap \mathbb{L}_I)^{wF}] \simeq R \oplus \text{other terms}$.
	\end{center}
	Thus, $\R \Gamma_c({\X}(I,wF),R) \simeq \R \Gamma_c(\tilde{\X}(I,wF),R) \otimes_{R \mathbb{L}_I^{wF}} R$ is a direct summand of $\R\Gamma_c(\X_{\mathbb{G}}(wF),R)$. Therefore  $\h^i(\X(I,wF),R)$ is a direct summand of $\h^i(\X(wF))$ and if $i<\ell(w)$ then $\h^i(\X(I,wF),R)=0$ (see \cite[Section 8]{Rouq-bonnDL}).
	\smallskip
	
	By the universal coefficient theorem and \ref{boun1}, if $w$ is $I$-reduced and $\ell \nmid |\bbL_I^{\dot{w}F}|$ then $\h^{\ell(w)}(\X(I,wF),\mathbb{Z}_\ell)$ is torsion-free over $\mathbb{Z}_\ell$.
\end{proof}

\subsection{$\Phi_d$-blocks} \label{not_sure} 
Assume that $\ell$ does not divide $q$ and $G$ is finite reductive group. Let $\ell \mid |G|$ and $\ell$ be large number, it divides a unique cyclotomic polynomial $\Phi_d(q)$. Furthermore, the Sylow $\ell$-subgroups of $G$ are abelian in that case \cite[Theorem 3.2]{l-sylow-abelian}. It turns out that many aspects of unipotent $\ell$-blocks depend only on $d$. In \cite{Malle-Michel-Broue}, Brou\'{e}, Malle and Michel developed a theory of generic blocks for unipotent blocks of finite reductive groups with respect to the integer $d$ rather than the integer $\ell$.

\begin{dfn}
	Let $d$ be a positive integer and  $\Phi_{d}$ be the corresponding cyclotomic polynomial;
	\begin{itemize}
		\item A \emph{$\Phi_{d}$-torus} is an $F$-stable algebraic torus $\mathbb{S}$ such that $| \mathbb{S}^{F} |= \Phi_{d}(q)^{a}$ for some non-negative integer $a$.
		\item A \emph{Sylow $\Phi_{d}$-subgroup} of $\mathbb{G}$ is a maximal $\Phi_{d}$-torus among  $\Phi_d$-tori.
		\item A \emph{$d$-Levi subgroup}  $\mathbb{M}$ of $\mathbb{G}$ is the centralizer of a $\Phi_{d}$-torus.
	\end{itemize}	
\end{dfn}
\begin{eg}
	When $d=1$, the $\Phi_1$-tori are the split tori and the $1$-split Levi subgroups are the $F$-stable complements of the $F$-stable parabolic subgroups. These are the Levi subgroups which are use in Harish-Chandra induction and restriction.
\end{eg}
\begin{thm}[{\cite{l-sylow-conjugate}}] \label{Malle1}
	Assume that   $\Phi_{d}(q)$ divides $ |G|$. Then there exists a Sylow $\Phi_d $-subgroup $\mathbb{T}$ such that $|T|=\Phi_{d}(q)^{a_d}$. Furthermore, every two Sylow $\Phi_d $-subgroups are conjugate under $G$. 
\end{thm}	
Recall from \ref{dlind.} that for every $F$-stable Levi subgroup $\bbL$ of $\bbG$ we have the  Deligne--Lusztig induction and restriction maps $\dli$ and $\dlr$ between the Grothendieck groups of $KL\mmod$ and $KG\mmod$.
\begin{dfn} \label{cupair1} \leavevmode
	\begin{itemize}	
		\item A unipotent character $\rho$ of $G$ is called \emph{$d$-cuspidal} if $\dlr(\rho)=0$ for every proper $d$-Levi subgroup $\bbL$ of $\mathbb{G}$.
		\item A \emph{$d$-cuspidal pair} is a pair $(\bbL,\rho)$ where $\bbL$ is a $d$-Levi subgroup and $\rho$ is a $d$-cuspidal character of $\bbL$.
		\item Given a cuspidal pair $(\bbL,\rho)$, the set of constituents of the virtual character $\dli(\rho)$ is called a $\Phi_d$-block.
	\end{itemize}
\end{dfn}	
If $\bbL$ is a $d$-Levi subgroup, we will denote by $(Z(\mathbb{L})^{\circ})_{\Phi_{d}}$ the unique  Sylow $\Phi_d $-subgroup of the torus $Z(\mathbb{L})^{\circ}$.
\begin{thm} [{\cite[Chapter 22]{Cabanes-rep-of-red}}]    \label{cupair2}
	Assume that $\ell > h$ (the Coxeter number). Let $d$ be the order of $q$ in $k^{\times}$. 
	There is a bijection
	$$(\bbL,\rho) \longmapsto B(\bbL,\rho) $$
	between the set of $d$-cuspidal pairs (up to conjugation) and the set of unipotent $\ell$-blocks of $G$ such that 
	\begin{itemize}
		\item The unipotent characters in $B(\mathbb{L},\rho)$ are exactly the irreducible constituents of $\R^{\mathbb{G}}_{\mathbb{L}}(\rho)$ (hence they form a $\Phi_d$-block).
		\item The defect group of the block $B(\mathbb{L},\rho)$ is  the Sylow $\ell$-subgroup of $((Z(\mathbb{L})^{\circ})_{\Phi_{d}})^{F}$.
		\item  Let $\mathbb{S}$ be a Sylow $\Phi_d $-subgroup of $\bbG$ and $\bbL=C_{\bbG}(\mathbb{S})$. Then the principal block is  $B(\mathbb{L},1_{L})$.
	\end{itemize}	
\end{thm}


\subsection{Modular representation theory of $\Gl_n(q)$}

\subsubsection{Partitions and $\beta$-sets}
A partition of $n$ is a non-increasing sequence of non-negative integers $(\lambda_1,\lambda_2,...,\lambda_r)$ such that $\lambda_1 \geq \lambda_2 \geq \cdots \geq \lambda_r$ and $\displaystyle \sum_{i} \lambda_{i}=n$.  A $\beta$-set of $\lambda$ is a finite set of decreasing numbers of the form $\{\lambda_1 +r -1 , \lambda_2 +r -2,\cdots,\lambda_r\}$.
\smallskip

Note that adding a zero in the partition $\lambda$ has the effect of changing the $\beta$-set
 to $\{\lambda_1 +r , \lambda_2+r-1. \cdots \lambda_r +1,0 \}$. In the sequel we will often add sufficiently many zeros to get large enough $\beta$-sets, see \ref{not_sure} . 
\smallskip

Let $d \in \{0,1,\cdots,n\}$ and $\mu$ be a partition of $n-d$.  Let $X = \{x_1 < x_2 < \cdots < x_s \}$ be a $\beta$-set of $\mu$ which is large enough such that it contains $\{0,1,...,d -1 \}$.  An addable $d$-hook is a pair $(x, x + d)$ where $x \in X$ and $x + d \notin X$. Adding a $d$-hook to $\mu$ corresponds to moving 
a bead one position left to an empty position on $d$-abacus (recall that the $r$-runner abacus is an abacus with $r$  horizontal runners, numbered by $ 0,\ldots,r -1$ from top to bottom. For each $j$, the runner $j$ has marked positions labeled by the non-negative integers congruent to $j$ modulo $r$ increasing down the runner). Removing a $d$-hook from the partition $n$ corresponds to moving 
a bead one position right to an empty position on $d$-abacus. We say that $\mu$ is a $d$-core if it has no $d$-hook. 
\smallskip

Let $X'$ be the subset of X defined by $X'=\{x \in x | x+d \notin X \} $. 
We denote by $\mu*x$ the partition which has the $\beta$-set $(X \setminus \{x\}) \cup \{x+d\}$.

\subsubsection{$\Phi_d$-blocks of $\Gl_n(q)$}
The unipotent characters of $\Gl_n(q)$ are labeled by partitions of $n$. Given a partition $\lambda$ of $n$ we will denote by $\nabla(\lambda)$ the unipotent character of $\Gl_{n}(q)$ corresponding to $\lambda$ with the convention that $\nabla(1,1,...,1) = \mathrm{St}_G$ is the Steinberg character of $\Gl_n(q)$. By abuse of notation it will also denote a chosen $\mathbb{Q}_\ell \Gl_n(q)$-module with character $\nabla(\lambda)$.
\smallskip

Set $\bbL_I= (\Gl_1(\overline{\mathbb{F}}_p))^d\times \Gl_{n-d}(\overline{\mathbb{F}}_p)$ and $\bbL_I^{\dot{w}F} = \Gl_1(q^d)\times \Gl_{n-d}(q)$. 
If $\mu$ is any partition of $n-d$ then it is shown in \cite{Malle-Michel-Broue} that  
$$ \R^{\bbG}_{\bbL_I}(\nabla(\mu)) = \sum_{x} \epsilon_x \nabla(\mu*x)$$
where $x$ is running over the addable $d$-hooks
of $\mu$ and $\epsilon_x=\pm 1$. Consequently, if $\mu$ is a $d$-core then the characters in  the $\Phi_d$-block $B(\bbL_I,\nabla(\mu))$ are labeled by the partitions which are obtained by adding one $d$-hook to $\mu$. 

\subsubsection{Modular unipotent representations of $\Gl_n(q)$}
 
By \cite{Dipper-J}, 
 the simple $\mathbb{F}_\ell \Gl_n(q)$-modules which lie in a unipotent block are also labeled by  partitions of $n$. More precisely, given $\lambda $ a partition of $n$ there exists a simple $\mathbb{F}_\ell \Gl_n(q)$-module $S_{\mathbb{F}_\ell}(\lambda)$  such that for every partition $\mu$ of $n$, the image of $\nabla(\mu)$ via decomposition map $\mathsf{dec}$ lies in the Grothendieck group of $\mathbb{F}_\ell \Gl_n(q)\mmod$ denoted by $K_0(\mathbb{F}_\ell \Gl_n(q)\mmod)$  and it is given by
\begin{equation}\label{eq:multiofcharacter}
\mathsf{dec}(\nabla(\mu))=[S_{\mathbb{F}_\ell}(\mu)] + \sum_{\mu \lhd \lambda} d_{\mu,\lambda} [S_{\mathbb{F}_\ell}(\lambda)]
\end{equation}
where  $d_{\mu,\lambda}$ are non-negative integer which are the entries of decomposition matrix of $\Gl_n(q)$ (page 28 of \cite{OLDL}). 
\smallskip

If $\nabla(\lambda)$ is in a block which has trivial defect (hence is alone in its block) then there is a unique lattice (i.e. a $\mathbb{Z}_\ell \Gl_n(q)$-module which is free as $\mathbb{Z}_\ell$-module) with character $\nabla(\lambda)$ and the $\ell$-reduction of this lattice is the simple module $S_{\mathbb{F}_\ell}(\lambda)$. It is no longer true if $\nabla(\lambda)$ lies in a block with non-trivial defect. However we can guarantee the uniqueness by imposing some conditions on the lattice, as shown in the following proposition.
\begin{prop}\label{nabla}
	Let $G=\Gl_n(q)$ and $\lambda$ be a partition of $n$. There exists a unique (up to isomorphism) $\mathbb{Z}_\ell G$-lattice $\nabla_{\mathbb{Z}_\ell}(\lambda)$ such that 
	\begin{enumerate}
		\item$\nabla_{\mathbb{Z}_\ell}(\lambda)$  has character $\nabla(\lambda)$.
		\item The socle of $\mathbb{F}_\ell \otimes_{\mathbb{Z}_\ell} \nabla_{\mathbb{Z}_\ell}(\lambda)$ is the simple module $S_{\mathbb{F}_\ell}(\lambda)$.
	\end{enumerate}	
\end{prop}	
\begin{proof}
	Let  $P_{\mathbb{F}_\ell}(\lambda)$ be a projective cover of $S_{\mathbb{F}_\ell}(\lambda)$ and $P_{\mathbb{Z}_\ell}(\lambda)$ be its unique lift as projective $\mathbb{Z}_\ell \Gl_n(q)$-module.
	The character of the module $P_{\mathbb{F}_\ell}(\lambda)$ i.e its image in the Grothendieck group lies in $K_0(\mathbb{Q}_\ell \Gl_n(q)\mmod)$. By Brauer reciprocity (see for example \cite[Proposition 4.1]{OLDL}) we have 
	\begin{equation*} e\big([P_{\mathbb{F}_l}(\lambda)]\big)= [ \mathbb{Q}_l \otimes_{\mathbb{Z}_l} P_{\mathbb{Z}_l}(\lambda) ]= \nabla(\lambda)+ \sum_{\mu \lhd \lambda} d_{\mu,\lambda} \nabla(\mu).
	\end{equation*}
	 By \eqref{eq:multiofcharacter}, $\nabla(\lambda)$ occurs with multiplicity one in the character of $P$. 
	Now, let $a_{\lambda}$ be the central idempotent of $\mathbb{Q}_{\ell}G$ associated with the irreducible character $\nabla(\lambda)$ (recall $a_{\lambda}=\frac{\dim \nabla(\lambda)} { \mid \Gl_n(q) \mid} \Sigma_{g \in G} \rho_{\lambda}(g) g^{-1} $). 
	We define a $\mathbb{Z}_\ell G$-lattice as a submodule of $P$ by: 
	\begin{center}
		$\nabla_{\mathbb{Z}_\ell}(\lambda):= a_{\lambda} P \bigcap P
		$.
	\end{center}
	This $\mathbb{Z}_\ell G$-lattice has irreducible character $\Delta(\lambda)$ over $\mathbb{Q}_\ell$
	\begin{center}
		$\mathbb{Q}_\ell \otimes_{\mathbb{Z}_\ell}  (a_{\lambda}P \bigcap P) \simeq \nabla(\lambda)$
	\end{center}	
	which proves (i). For proving (ii) we consider 
	the short exact sequence as follows
	\[
	\xymatrixcolsep{0.6cm}
	\xymatrixrowsep{0.6cm}
	\xymatrix{0 \ar[r] & \nabla_{\mathbb{Z}_\ell}(\lambda) \ar[r] & P \ar[r] & P / \nabla_{\mathbb{Z}_\ell}(\lambda) \ar[r] & 0}.
	\]
	Tensoring by $\mathbb{F}_\ell$ yields the following exact sequence
	\begin{small}
		\[
		\xymatrixcolsep{0.32cm}
		\xymatrixrowsep{0.34cm}
		\xymatrix{ 0 \ar[r] & \tor^{1}_{\mathbb{Z}_\ell}(P / \nabla_{\mathbb{Z}_\ell}(\lambda),\mathbb{F}_\ell)  \ar[r]  & \nabla_{\mathbb{F}_\ell}(\lambda) \otimes_{\mathbb{Z}_\ell} \mathbb{F}_\ell  \ar[r] & P \otimes_{\mathbb{Z}_\ell} \mathbb{F}_\ell\ar[r] &  P / \nabla_{\mathbb{Z}_\ell}(\lambda) \otimes_{\mathbb{Z}_\ell} \mathbb{F}_\ell \ar[r] & 0}.
		\]
	\end{small}
	If $\tor^{1}_{\mathbb{Z}_\ell}(P / \nabla_{\mathbb{Z}_\ell}(\lambda),\mathbb{F}_\ell)=0$ then we have 
	\begin{center}
		$ \nabla_{\mathbb{F}_\ell}(\lambda) \otimes_{\mathbb{Z}_\ell} \mathbb{F}_\ell \hookrightarrow P \otimes_{\mathbb{Z}_\ell} \mathbb{F}_\ell\ \simeq  P_{\mathbb{F}_\ell}(\lambda)$.
	\end{center}	
	We shall prove that $P / \nabla_{\mathbb{Z}_\ell}(\lambda)$  is a torsion-free module over $\mathbb{Z}_\ell$.  For simplicity,  let us set $e:=a_{\lambda}$. Let us consider $x \in P$ such that $rx \in P \cap eP$ for some $r \in \mathbb{Z}_\ell$. We have  $rx=ey$ for some $y \in P$. Then 
	\begin{center}
		$erx=e\cdot ey=e^2 y=ey=rx$.
	\end{center}	
	Hence $r(ex)=rx$. Since $P$ is a torsion-free module, we deduce that $r(x-ex)=0$ and that $x=ex$. So $x$ should be in $P \bigcap eP$. This shows  that $ \nabla_{\mathbb{F}_\ell}(\lambda) \otimes_{\mathbb{Z}_\ell} \mathbb{F}_l$ embeds in  $P \otimes_{\mathbb{Z}_\ell} \mathbb{F}_\ell\ \simeq  P_{\mathbb{F}_\ell}(\lambda)$. Since the socle of  $P_{\mathbb{F}_\ell}(\lambda)$ is $S_{\mathbb{F}_\ell}(\lambda)$, then so is the socle of  $\nabla_{\mathbb{F}_\ell}(\lambda) \otimes_{\mathbb{Z}_\ell} \mathbb{F}_\ell$ which proves (ii).
	
	\smallskip 
	It is left to prove the uniqueness. For simplicity of the notation we use the notation $\nabla$ instead of $\nabla_{\mathbb{Z}_\ell}(\lambda)$. Let us consider a $\mathbb{Z}_\ell G$-lattice $\nabla^{\prime}$ satisfying the conditions (i) and (ii). We shall prove that $\nabla\simeq \nabla^{\prime}$. 
	We claim that the map $\Hom_{\mathbb{Z}_\ell G\mmod}(\nabla^{'},P) \otimes_{\mathbb{Z}_\ell} \mathbb{F}_\ell \longrightarrow \Hom_{kG\mmod}(\mathbb{F}_\ell \nabla^{'},\mathbb{F}_\ell P)$ is an isomorphism of $\mathbb{F}_\ell $-vector spaces. By naturality, we only need to show it when $P=\mathbb{Z}_\ell G$.  Since $\mathbb{F}_\ell G$ and $\mathbb{Z}_\ell G$ are symmetric algebras then we have the following diagram 
	\[
	\xymatrix{\Hom_{(\mathbb{Z}_\ell G\mmod)}(\nabla^{'},\mathbb{Z}_\ell G) \otimes_{\mathbb{Z}_\ell }\mathbb{F}_\ell  \ar[r]^{\hspace{0.16cm} h^{'}} \ar[d]^{\simeq}  & \Hom_{(\mathbb{F}_\ell G\mmod)}(\mathbb{F}_\ell \nabla^{'},\mathbb{F}_\ell G) \ar[d]^{\simeq} \\
		\Hom_{(\mathbb{Z}_\ell \mmod)}(\nabla^{'},\mathbb{Z}_\ell ) \otimes_{\mathbb{Z}_\ell }\mathbb{F}_\ell  \ar[r]^{h} &  \Hom_{(\mathbb{F}_\ell \mmod)}(\mathbb{F}_\ell \nabla^{'},\mathbb{F}_\ell )}
	\]
	Since $\nabla^{'}$ is $\mathbb{Z}_\ell $-free module of finite rank then $h$ is an isomorphism. Therefore $h'$ is an isomorphism as claimed.
	
	\smallskip
	We deduce that there exists a map $f:\nabla^{'} \longrightarrow P$ such that $\bar{f}=f \otimes_{\mathbb{Z}_\ell }\mathbb{F}_\ell  $ is non-zero. We note that $\mathbb{F}_\ell  \nabla^{'}$ and $P_{\mathbb{F}_\ell}(\lambda)$ have the same socle therefore $\bar{f}$ is injective. Now we prove that $f$ is also injective. We have $\kerr(f) \subseteq \nabla^{'}$. Since $\mathbb{Q}_\ell \nabla^{'}$ is simple, $\kerr(f)=0$ and hence $f$ is injective. Now we can write the short exact sequence 
	\[
	\xymatrixcolsep{0.6cm}
	\xymatrixrowsep{0.6cm}
	\xymatrix{0 \ar[r] & \nabla^{\prime} \ar[r]^{f} & P \ar[r] & P / f(\nabla^{\prime}) \ar[r] & 0}.
	\]
	This yields the exact sequence 
	\[
	\xymatrixcolsep{0.6cm}
	\xymatrixrowsep{0.6cm}
	\xymatrix{0 \ar[r] &  \tor^{1}_{\mathbb{Z}_\ell}(P/f(\nabla^{'}),\mathbb{F}_\ell)  \ar[r] &  \mathbb{F}_\ell \nabla^{\prime} \ar[r]^{\bar f} & \mathbb{F}_\ell  P \ar[r] & \mathbb{F}_\ell P / \mathbb{F}_\ell  f(\nabla^{\prime}) \ar[r] & 0}.
	\]
	Since $\bar{f}$ is injective, $P/f(\nabla^{'})$ is a torsion-free module.  
	Now, we will prove that $\nabla \simeq \nabla^{'}$. We proved that both $P/\nabla$ and $P/f(\nabla^{'})$ are torsion-free modules. The modules $\mathbb{Q}_\ell f(\nabla^{'})$ and $\mathbb{Q}_\ell \nabla$ are two pure submodules of  $\mathbb{Q}_\ell P$ and the character $\nabla(\lambda)$ only occurs with multiplicity one, therefore $\mathbb{Q}_\ell f(\nabla^{'})=\mathbb{Q}_\ell \nabla$. We also have $ \nabla \subseteq
	f(\nabla^{'})$. If $f(\nabla^{'}) \oplus A =P$ and $\nabla \oplus B =P$ and $x \in \nabla$ then we can write $x=a_1 +a_2$ with $a_1 \in f(\nabla^{'})$ and $a_2 \in A$. There exists $\theta \in \mathbb{Q}_\ell$ such that $\theta x \in \nabla$, but $\theta x=\theta a_1 + \theta a_2$ forces $a_2=0$. Hence $x=a_1 \in f(\nabla^{'})$, this shows that $\nabla \subseteq f(\nabla^{'})$. With the same argument, we can show that $f(\nabla^{'}) \subseteq \nabla$ and therefore $f(\nabla^{'})= \nabla$. We deduce that $f$ induces an isomorphism  $\nabla' \simto \nabla$.
\end{proof}	
\begin{notation}
	If $R$ is any ring among $\ell$-modular system $(\mathbb{Q}_\ell,\mathbb{Z}_\ell,\mathbb{F}_\ell)$ we define
	$$\nabla_R(\lambda):=\nabla_{\mathbb{Z}_\ell}(\lambda) \otimes_{\mathbb{Z}_\ell} R.$$
	With this notation we have $\nabla_{\mathbb{Q}_\ell}(\lambda) \simeq \nabla(\lambda)$.
\end{notation}	


\section{Cohomology of the Deligne--Lusztig variety $\X_{n,d}$} \label{def:Xnd} 
\subsection{Preliminaries}
We start to state some general properties of the cohomology of parabolic Deligne--Lusztig varieties in the case of $\Gl_n(q)$.
\begin{thm} \label{1,st1}
	Let $G=\Gl_n(q)$.  Let $(I,wF)$ be such that $w \in \mathfrak{S}_n$ is $I$-reduced and ${}^{w}I=I$.
	\begin{enumerate}
		\item The trivial representation $\nabla(n)$ only occurs in $\h^{i}_c(\X(I,wF),\mathbb{Q}_\ell)$ for $i=2\ell(w)$ and it occurs with multiplicity one. 
		\item  The Steinberg representation  $\nabla(1^n)$ can only occur in $\h^{i}_c(\X(I,wF),\mathbb{Q}_\ell)$ if $I=\emptyset$ and $i=\ell(w)$. In that case it occurs with multiplicity one.
		\item Assume  $\ell \nmid |\mathbb{L}_I^{\dot{w}F}|$. The simple module  $S_{\mathbb{F}_\ell}(1^n)$ can only be a composition factor of $\h^{i}_c(\X(I,wF),\mathbb{F}_\ell)$ where $i=\ell(w)$.
	\end{enumerate}
\end{thm}
\begin{proof}
	The properties (i) and (ii) are proved in \cite[Corollary 8.38]{parabolicDM}. Property (iii) is proved in \cite{Gelfand} in the case where $I=\emptyset$.  Using a same argument as in the proof of \ref{boun1}, we can generalize it to any $I$.
\end{proof}	
By \cite{cusp_gln}, a unipotent simple $\mathbb{F}_\ell \Gl_n(q)$-module $S_{\mathbb{F}_\ell}(\lambda)$ is cuspidal if and only if $\lambda=(1^n)$ and $n=m$ where $m$ is the order of $q$ in $\mathbb{F}_\ell^{\times}$.
\begin{cro} \label{middle_degreee}
	Assume  $\ell \nmid |\mathbb{L}_I^{\dot{w}F}|$. Cuspidal $\mathbb{F}_\ell \Gl_n(q)$-modules can only occur as a composition factor in the middle degree of the cohomology of $\X(I,wF)$.
\end{cro}

\subsection{The Deligne--Lusztig variety $\X_{n,d}$} 
Recall from \cite[Theorem 1.2]{Digne-Michel} that $(W,S)$ is a Coxeter system.  Let $\mathbf{S}$ be a set in bijection with $S$. Then we can define the Artin-Tits braid monoid by the following presentation
\begin{align} \label{tit}
B_{W}^{+}=\langle \mathbf{s} \in \mathbf{S} \mid  \underbrace{\mathbf{sts} \cdots}_{m_{s,t}\emph{\text{ terms}}}  =\underbrace{\mathbf{tst} \cdots}_{m_{s,t}\emph{\text{ terms}}}  \rangle_{\text{mon}}.
\end{align}
Following \cite{parabolicDM}, we can extend the definition of parabolic Deligne--Lusztig varieties to the elements of the braid monoid. Let $\mathbf{I} \subset \mathbf{S}$  and $\mathbf{b} \in B^{+}_{W}$ such that 
\begin{itemize}
	\item $\mathbf{b}$ has no prefix in $B^{+}_{W_I}$(analog of being $I$-reduced) and
	\item for all $\mathbf{s} \in \mathbf{I}$, $\mathbf{b}F(\mathbf{s})\mathbf{b}^{-1} \in  \mathbf{I}$.
\end{itemize}
Then there is a corresponding parabolic Deligne--Lusztig variety $\X(\mathbf{I},\mathbf{b}F)$ such that  $\X(\mathbf{I},\mathbf{b}F) \simeq \X(I,wF)$ whenever $\mathbf{b}$ is the lift to $B^{+}_{W}$ of an element $w \in W$. We already defined the parabolic Deligne--Lusztig varieties $\X(I,wF)$ and $\tilde{\X}(I,\dot{w}F)$  associated to a pair $(I,wF)$, where $w \in W$  is $I$-reduced and  $^{wF} I =I$. Here, we define the variety $\X_{n,d}$, the variety whose cohomology we aim to study, using a specific pair $(I,w)$. We mentioned that the pair $(I,w)$ can be extend to general case $(\mathbf{I},\mathbf{b})$. In \cite{cohomology-of-DL_Dudas}, for a special pair $(\mathbf{I},\mathbf{b})$, the corresponding Deligne--Lusztig varieties are defined for $\Gl_n(q)$ for studying $\Phi_d$-blocks of $\Gl_n(q)$. This definition is due to \cite{cohomology-of-DL_Dudas} and \cite{parabolicDM}.
\begin{dfn}
	Let $1 \leq d \leq n$. Consider the following pairs
	\begin{align*}
	&& \mathbf{v}_d=\mathbf{s}_1\mathbf{s}_2 \ldots  \mathbf{s}_{n-1-[d/2]} \mathbf{s}_{n-1}\mathbf{s}_{n-2} \ldots \mathbf{s}_{[d+1/2]} \in B^{+} &&\\
	\text{and} && \mathbf{J}_d=\{\mathbf{s}_i \mid [d+1/2]+1 \leq i \leq n-1-[d/2]\} \subset \mathbf{S}.
	\end{align*}		
	Then 
	$$\X_{n,d}:=\X(\mathbf{J}_{d},\mathbf{v}_d F).$$
	The Galois covering $\X_{n,d}$ will be denoted by $ \tilde{\X}_{n,d}$
\end{dfn}	
Note that for $d > 1$, the element $\mathbf{v}_d$ is reduced. Therefore, it is the lift of an element $v_d \in W$ and therefore we can work with the variety $\X(J_d,v_d F)$ instead. By \cite[Lemma 11.7, 11.8]{parabolicDM}, the pair $(\mathbf{J}_d,\mathbf{v}_d)$ is $d$-periodic
 so that it makes sense to study the cohomology of $\X_{n,d}$. 
\smallskip

The group $\Gl_n(q)$ acts on the Deligne--Lusztig variety $\tilde{\X}_{n,d}$ by left multiplication and the Levi subgroup  $\mathbb{L}_{I}^{\dot{w}F}  \simeq \Gl_1(q^d) \times \Gl_{n-d}(q)$ acts by right multiplication. Also $\Gl_n(q)$ acts on $\X_{n,d}$ by left multiplication.

\subsection{Cohomology of $\X_{n,d}$ over $\mathbb{Q}_\ell$}

Let $\mu$ be a partition of $n-d$ with 
 the corresponding unipotent  character $\nabla(\mu)$ of $\Gl_{n-d}(q)$. It defines a local system on the variety $\X_{n,d}$ such that 
\begin{center}
	$\h^{\bullet}_c(\X_{n,d},\nabla(\mu)) \simeq \h^{\bullet}_c(\tilde{\X}_{n,d},\mathbb{Q}_\ell)_{\nabla(\mu)}$. 
\end{center}	
\begin{thm} [{\cite[Theorem 2.1]{cohomology-of-DL_Dudas}}] \label{localsys}
	Let $\mu$ be a partition of $n-d$. Write 
	$$ ^*\! \R^{\Gl_{n-d}} _{\Gl_{n-d-1}}(\nabla(\mu)) = \sum_i \nabla(\mu^{(i)})$$
	where the $\mu^{(i)}$'s are partitions of $n-d-1$. Then there exists a distinguished triangle in $D^{b}(\mathbb{Q}_\ell \Gl_{n-1}(q) \times \langle F\rangle\mmod)$
	\begin{small}
		\[
		\xymatrixcolsep{0.4cm}
		\xymatrixrowsep{0.4cm}
		\xymatrix{\R\Gamma_c(\mathbb{G}_m \times \X_{n-1,d-1},\nabla(\mu))  \ar[r] & {}^*\mathscr{R}_{\Gl_{n-1}}^{\Gl_n} \big( \R \Gamma_c(\X_{n,d},\nabla(\mu)) \big)\ar[r] & \displaystyle
			\bigoplus_i \R \Gamma_c(\X_{n-1,d},\nabla(\mu^{(i)})) [-2](1) \ar@{~>}[r] & }	\] 
	\end{small}
	where (1) is the Tate twist.
\end{thm}
In \cite{cohomology-of-DL_Dudas}, a formula is given for computing the cohomology groups of $\X_{n,d}$ with coefficients in $\nabla(\mu)$. 
The $\nabla(\lambda)$'s occurring in $\h^{\bullet}(\X_{n,d},\nabla(\mu))$ are associated to the partitions obtained from $\mu$ by adding a $d$-hook.   
\begin{thm}[{\cite{cohomology-of-DL_Dudas}}] \label{n-d} \label{Xnd1}
	Let $\mu$ be a partition of $n-d$ and $X$ be a $\beta$-set of $\mu$  and  let $X^{\prime}=\{x \in X \mid x+d \notin X  \}$. Given $x \in X'$, define
	\begin{center}
		$\pi_{d}(X,x)=2(n-1+x - \# \{y \in X \mid y< x \}) - \# \{y \in X \mid x <y<x+d \}$
	\end{center}	
	and 
	\begin{center}
		$\gamma_d(X,x)=n+x- \# X$.
	\end{center}	
	Then $\nabla(\mu*x)$ occurs in $\h^{i}_c(\X_{n,d},\nabla(\mu))$ for $i=\pi_{d}(X,x)$ only. Furthermore, it occurs with multiplicity one and with eigenvalue of $F$ equal to $q^{\gamma_d(X,x)}$. In other words, 
	$$\R\Gamma_{c}(\X_{n,d},\nabla(\mu))\simeq \bigoplus_{x \in X^{\prime}} \nabla(\mu *x)[-\pi_{d}(X,x)] \langle q^{\gamma_d(X,x)}\rangle $$
	in  $D^b(\mathbb{Q}_\ell \Gl_n(q) \times \langle F\rangle )$-modules.
\end{thm}


Let $\mu$ be the trivial partition $(n-d)$. Then $\nabla(\mu)$ is the trivial representation and corresponds to the trivial local system $\mathbb{Q}_{\ell}$ on the variety $\X_{n,d}$  and its restriction  $^*\! \R^{\Gl_{n-d}} _{\Gl_{n-d-1}}(\nabla(\mu))$ corresponds to the trivial local system on $\X_{n-1,d}$. Consequently \ref{localsys} can be written as follows in that specific case.
\begin{cro} 
	There is a distinguished triangle in $D^b(\mathbb{Q}_\ell \Gl_{n-1}(q) \times \langle F\rangle\mmod)$
	\begin{small}
		\[
		\xymatrixcolsep{0.5cm}
		\xymatrixrowsep{0.5cm}
		\xymatrix{\R\Gamma_c(\X_{n-1,d-1} \times \mathbb{G}_m,\mathbb{Q}_\ell) \ar[r] &  ^* \! \mathscr{R}^{\Gl_{n}}_{\Gl_{n-1}} \R \Gamma_c(\X_{n,d},\mathbb{Q}_\ell) \ar[r] & 
			\R \Gamma_c(\X_{n-1,d},\mathbb{Q}_\ell )[-2](1) \ar@{~>}[r] & }.	\] 
	\end{small}
\end{cro}
We are now going to describe explicitly the various invariants introduced before in the case where $\mu$ is the trivial partition $(n-d)$. We take $X=\{n,d-1,\cdots,1,0 \}$ to be a $\beta$-set for $\mu$ by adding $d$ zeros to the partition $(n-d)$. It is large enough for adding and removing $d$-hooks. Depending on $n$ and $d$, there are two possibilities for the set $X^{\prime}$ of addable $d$-hooks
\begin{equation}
X^{\prime}=\begin{cases}
X & \text{if }  n-d \geq d,\\
X \setminus \{n-d\} &  \text{otherwise}.
\end{cases}
\end{equation} \label{Xprime}
Let $x \in X'$. We have 
\begin{equation*} 
\# \{y \in X \mid y< x \}=\begin{cases}
d & \text{if }   x=n,\\
x & \text{if }  x<d,
\end{cases}
\end{equation*}
and
\begin{equation*} 
\# \{y \in X \mid x <y<x+d \}=\begin{cases}
0 & \text{if }  x=n,\\
d-x-1 & \text{if }   x \leq  n-d,\\
d-x & \text{if }  n-d<x<n.
\end{cases}
\end{equation*}
We deduce that 
\begin{equation} 
\pi_{d}(X,x)=\begin{cases}
2(2n-d-1) & \text{if }   x=n,\\
2n-1-d+x & \text{if }   x \leq  n-d,\\
2n-2-d+x & \text{if }   n-d <x <n.
\end{cases}
\end{equation}
Also $\gamma_d(X,x)=n+x-d-1$.
\smallskip

For $x \in X^{\prime}$, recall that $\mu*x$ is the partition of $n$ obtained from $\mu$ by adding the $d$-hook $(x,x+d)$. Depending on $x$ it is given by
\begin{equation} \label{mu}
\mu*x=\begin{cases}
(n) & \text{if }   x=n,\\
(n-d,x+1,1^{d-x-1}) & \text{if }   x  < n-d \text{ and } x\leq d-1, \\
(x,n-d+1,1^{d-x-1}) &\text{if }   n-d \leq x < d-1,\\
0 & \text{otherwise}.
\end{cases}
\end{equation}
This can be seen using  the $1$-abacus of $\mu$. For example let us consider the case where $x < n-d$ and $x \leq d-1$. The $1$-abacus corresponding to $x$ for the trivial partition $(n-d) \vdash n-d$ is
\begin{small}
	\xymatrixcolsep{0,7cm}
	\[\xymatrix{
		\cdots \ar@{-}[r] & \nod{} \ar@{-}[r] &	\node{n} \ar@{-}[r]  & \nod{n-1} \ar@{-}[r] &\nod{n-2} \ar@{-}[r]  & \nod{} \cdots \hspace{-1cm} & \nod{}  \ar@{-}[r] & \node{d-1} \cdots \hspace{-1cm} &\node{x+1} \ar@{-}[r]  & \node{x} \cdots \hspace{-1cm} & \node{2}  \ar@{-}[r]  & \node{1} \ar@{-}[r] & \node{0}
	}\]
\end{small} 
\noindent The black bead $x$ is located  between $0$ and $d-1$. Adding the $d$-hook corresponding to $x$ amounts to move it $d$ steps to the left. It will be between $d-1$ and $n$. 
In that case the number of white beads between $n$ and $x+d$ is $n-(x+d)-1$. The number of white beads between $x+d$ and $d-1$ is $(x+d)-(d-1)-1=x$. Finally, there is one white bead at $x$. Recall that the number of white beads to the right side of each black bead gives us the partition $\lambda$. The number of white beads in the right side of the bead $n-1$ is $(n-(x+d)-1) + x + 1 = n-d$; the number of white beads to the right of $x+d$ is $x+1$. Finally, there is only at most one white bead to the right of the remaining black beads. This gives
\begin{center}
	$\lambda=(n-d,x+1,1^{d-x-1})$.
\end{center}	
\smallskip

We can now spell out \ref{Xnd1}  for the trivial partition $\mu=(n-d)$ using the previous computations. This makes the  cohomology groups of the Deligne--Lusztig variety $\X_{n,d}$ over $\mathbb{Q}_\ell$ explicit.
\begin{thm} \label{Xnd}
	Let  $ 1\leq d \leq n$. Then the cohomology of $\X_{n,d}$ over $\mathbb{Q}_\ell$ is explicitly given by 
	$$
	\begin{small}
	\begin{array}{c|cccccc} i & 2n-1-d & 2n-d & 2n-d+1 & \cdots &  3n-2d-2  \\\hline \\[0.08cm]
	\h^{i}_c(\X_{n,d},\mathbb{Q}_\ell)& \nabla(n-d,1^d)  & \nabla(n-d,21^{d-2})  & \nabla(n-d,31^{d-3}) & \cdots &\nabla((n-d)^2,1^{2d-n})   \\[0.3cm]
	F & q^{n-d-1} & q^{n-d} & q^{n-d+1} & \cdots & q^{2n-2d-2}&\\[0.3cm]
	\end{array}
	\end{small}
	$$
	\\
	$$
	\begin{small}
	\begin{array}{c|ccc}
	i & 3n-2d-1 &  3n-2d \\\hline \\[0.08cm]
	\h^{i}_c(\X_{n,d},\mathbb{Q}_\ell) &\nabla((n-d+1)^2,1^{2d-n-2}) & \nabla(n-d+2,n-d+1,1^{2d-n-3}) \\[0.3cm]
	F &q^{2n-2d}  & q^{2n-2d+1}   \\[0.3cm]
	\end{array}
	\end{small}
	$$
	\\
	$$
	\begin{small}
	\begin{array}{c|cccc}
	i &  3n-2d+1& \cdots& 2n-4  \\\hline \\[0.08cm]
	\h^{i}_c(\X_{n,d},\mathbb{Q}_\ell) &\nabla(n-d+3,n-d+1,1^{2d-n-4})  & \cdots & \nabla(d-2,n-d+1,1) \\[0.3cm]
	F &q^{2n-2d+2} & \cdots  & q^{n-3}  & \\[0.3cm]
	\end{array}
	\end{small}
	$$
	\\
	$$
	\begin{small}
	\begin{array}{c|ccccc}
	i & 2n-3 & 2n-2 & \cdots &4n-2d-2 \\\hline \\[0.08cm]
	\h^{i}_c(\X_{n,d},\mathbb{Q}_\ell)  & \nabla(d-1,n-d+1) & 0& 0 & \nabla(n) \\[0.3cm]
	F  &  q^{n-2}  & 0 & 0  & q^{2n-1-d} \\[0.3cm]
	\end{array}
	\end{small}
	$$
\end{thm}	
We note that the first zero term occurs in degree $2n-2$ and the last zero term occurs in degree  $4n-2d-3$. The number of zero terms in the table is $2n-2d$. In particular, when $n=d$ only (the Coxeter case) all the cohomology groups between middle and top degrees are non-zero. 
\smallskip

For the sake of the reader, we also give the cohomology groups  of the Deligne--Lusztig variety $\X_{n,d}$ in terms of $ \mu * x$. We distinguish the cases depending on what $X'$ is.

\noindent
$\bullet$ Assume that $n+1<2d$. Then for 
$x \in \{d-2,\dots,2,1,0\} \cup \{2n-d-1\}$:
\begin{equation}\label{mu1}
\h_c^{2n-d-1+x}(\X_{n,d},\mathbb{Q}_\ell)=\begin{cases}
\nabla(\mu * x) \langle q^{n-d-1+x} \rangle &  x<n-d,\\
\nabla(\mu * (x+1)) \langle q^{n-d+x}\rangle &  n-d\leq x < d-1, \\
\nabla(n) \langle q^x\rangle & x=2n-d-1, \\
0 &\text{otherwise}.
\end{cases}
\end{equation}
\noindent
$\bullet$ Assume that $n +1 >2d$. Then $x \in \{d-1,\dots,2,1,0\} \cup \{2n-d-1\}$
\begin{equation}
\h_c^{2n-d-1+x}(\X_{n,d},\mathbb{Q}_\ell)=\begin{cases}
\nabla(\mu * x) \langle q^{n-d-1+x}\rangle &  x\leq d-1, \\
\nabla(n) \langle q^x \rangle & x=2n-d-1,
\\
0 & \text{otherwise}.
\end{cases}
\end{equation}
\noindent
$\bullet$ Assume that $n+1=2d$. Then for $x \in \{d-2,\dots,2,1,0\}\cup \{2n-d-1\}$
\begin{equation}
\h_c^{2n-d-1+x}(\X_{n,d},\mathbb{Q}_\ell)=\begin{cases}
\nabla(\mu * x) \langle q^{n-d-1+x}\rangle  &  x\leq d-2,\\
\nabla(n) \langle q^x\rangle & x=2n-d-1,\\
0 & \text{otherwise}.
\end{cases}
\end{equation}
One can actually  group the three cases together as follows:
	For $x \in \{d-1,\ldots,2,1,0\} \cup \{2n-d-1\}$,
	$$  \h_c^{2n-d-1+x}(\X_{n,d},\mathbb{Q}_\ell)=\begin{cases}
	\nabla(\mu * x) \langle q^{n-d-1+x}\rangle  &  x<n-d, \text{ and } x\leq d-1, \\
	\nabla(\mu * (x+1)) \langle q^{n-d+x}\rangle  &  n-d\leq x < d-1,  \\
	\nabla(n) \langle q^x \rangle  & x=2n-d-1,\\
	0  & \text{otherwise}.
	\end{cases} $$

\subsection{Cohomology of $\protect \xnd$ over $\mathbb{Z}_\ell$}
We will determine explicitly  the cohomology groups of the Deligne--Lusztig varieties $\X_{n,d}$ over $\mathbb{Z}_\ell$. We will work under the assumption that $\ell \mid \Phi_m(q)$ where $m>n-d+1$ and $m>d$ and $m>6$. In that situation the cohomology complex is perfect and we will show that its cohomology is torsion-free.
\subsubsection{Statement of the main results}

We start by extending \ref{localsys} to any ring $R$ among the modular system $(\mathbb{Q}_\ell,\mathbb{Z}_\ell,\mathbb{F}_\ell)$:
\begin{thm}\label{Cnd}  \label{generallong}
	There is a distinguished triangle in $D^b(R \Gl_{n-1}(q) \times \langle F\rangle\mmod)$
	\begin{small}
		\[
		\xymatrixcolsep{0.5cm}
		\xymatrixrowsep{0.5cm}
		\xymatrix{\R\Gamma_c(\X_{n-1,d-1} \times \mathbb{G}_m,R) \ar[r] &  ^* \! \mathscr{R}^{\Gl_{n}}_{\Gl_{n-1}} \R \Gamma_c(\X_{n,d},R) \ar[r] & 
			\R \Gamma_c(\X_{n-1,d},R )[-2](1) \ar@{~>}[r] & }.	\] 
	\end{small}
\end{thm}
\begin{proof}
	Let $\bbU$ be the unipotent radical of an $F$-stable parabolic subgroup of $\bbG$ with Levi complement $\Gl_{n-d}(\overline{\mathbb{F}}_p)\times \Gl_1(\overline{\mathbb{F}}_p)$. We also set $L = \Gl_{n-1}(q)\times\Gl_1(q^d)$. 
	We follow the proof of \cite[Theorem 2.1]{cohomology-of-DL_Dudas}. We can decompose the variety $\tilde{\X}_{n,d}$ as $\tilde{\X}_{n,d}=\tilde{\X}_{z_0} \sqcup \tilde{\X}_{z_1}$ where $\tilde{\X}_{z_0}$ is an open subvariety of $\tilde{\X}_{n,d}$ stable by the action of $U$ on the left and $L$ on the right. By  \cite[Proposition 3.2]{quotient} we have 
	$$ \rgc(U \backslash \tilde{\X}_{z_1},R) \simeq \rgc(\tilde{\X}_{n-1,d}/U',R)[-2](1) $$ 
	where $U'$ is some unipotent  subgroup of $L$. Since $L$ is an $\ell'$-group we deduce that 
	$$ \rgc(U \backslash \tilde{\X}_{z_1}/L,R) \simeq \rgc(\X_{n-1,d},R)[-2](1).$$ 
	For the second piece we use \cite[Proposition 3.2]{quotient}. There exists $N \subset L$ and $N' \subset L':=\Gl_{n-d}(q)\times\Gl_1(q^{d-1})$ such that  
	$L/N \simeq L'/N'$ and
	$$ \rgc(U \backslash \tilde{\X}_{z_0},R) \simeq \rgc(\bbG_m \times  \tilde{\X}_{n-1,d-1}/N',R).$$ 
	Moding out by $L$ we obtain 
	$$ \rgc(U \backslash \tilde{\X}_{z_0}/L,R) \simeq \rgc(\bbG_m \times  \X_{n-1,d-1},R).$$ 
	We conclude using open-closed theorem 
	 for the decomposition 
	$U\backslash\X_{n,d}=U\backslash\tilde{\X}_{z_0}/L \sqcup U\backslash\tilde{\X}_{z_1}/L$ and the fact that $\sr$ is given by taking the fixed points under $U$.
\end{proof}	
Since distinguished triangles induce long exact sequences by open-closed theorem
 we obtain the following corollary which we are going to use to compute inductively the cohomology.
\begin{cro} \label{dm2}
There is a long exact sequence
\begin{small}
		\[
		\xymatrixcolsep{0.27cm}
		\xymatrixrowsep{0.27cm}
		\xymatrix{
			\cdots \ar[r] & \h_c^{i-1}(\X_{n-1,d-1},R) \oplus \h_c^{i-2}(\X_{n-1,d-1},R)(1) \ar[r] & ^* \! \mathscr{R}^{\Gl_n}_{\Gl_{n-1}}(\h_c^{i}(\X_{n,d},R)) \ar[r] & \h_c^{i-2} (\X_{n-1,d},R)(1) \ar `r[d]  `[l]  `[dlll] `[ll] [lld] \\
			& \h_c^{i}(\X_{n-1,d-1},R) \oplus \h_c^{i-1}(\X_{n-1,d-1},R) (1) \ar[r] & ^* \! \mathscr{R}^{\Gl_{n}}_{\Gl_{n-1}}(\h_c^{i+1}(\X_{n,d},R)) \ar[r] & \h_c^{i-1}(\X_{n-1,d},R)(1) \ar[r] & \cdots}
		\]
	\end{small}
\end{cro}
For simplicity, we will work with a shifted complex  instead of the cohomology complex of $\X_{n,d}$. We define
\begin{center}
	$C_{n,d}^{R} :=\R \Gamma_c(\X_{n,d},R)[-2n+d+1](-n+d+1)$.
\end{center}	
Then \ref{Cnd} translates into the following distinguished triangle in $D^b(R\Gl_{n-1}(q) \times \langle F\rangle\mmod)$
\[
\xymatrix{^* \! \mathscr{R}^{\Gl_n}_{\Gl_{n-1}} (C_{n,d}^R) \ar[r] & C_{n-1,d}^R \ar[r] & C_{n-1,d-1}[1] \oplus C_{n-1,d-1}^R(1) \ar@{~>}[r] & }.
\]
In addition,  we can rewrite theorem \ref{Xnd} for $C_{n,d}^{\mathbb{Q}_\ell}$. Namely if  $\mu=(n-d)$  is the trivial partition of  $n-d$ and  $x \in \{d-1,\cdots,2,1,0\} \cup \{2n-d-1\}$, we have 
\begin{equation}
\h^{x}(C_{n,d}^{\mathbb{Q}_\ell})=\begin{cases}
\nabla(\mu * x) \langle q^{x}\rangle  &  x<n-d \text{ and } x\leq d-1, \\
\nabla(\mu * (x+1)) \langle q^{x+1}\rangle  &  n-d\leq x < d-1, \\
\nabla(n) \langle q^n \rangle  & x=2n-d-1,\\
0 &\text{otherwise}.
\end{cases}
\end{equation}

The strategy  for computing the cohomology of $C_{n,d}^R$ is to break it into the various generalized eigenspaces of $F$ and to study each summand separately. In the case where $\ell \mid \Phi_m(q)$ with $m>d$ and $m>n-d+1$ and $m>6$ then we will see that each summand corresponds to different block. The most complicated summand is the one corresponding to the principal block which under our assumption on $\ell$ is the only unipotent block with non-trivial cyclic defect group. For this one we shall use the explicit knowledge of the tree.

\begin{thm} \label{expXnd}
	Let $R$ be any ring among $\mathbb{Z}_\ell$  and $\mathbb{F}_\ell$. Assume $\ell \mid \Phi_m(q)$ with $m>d$ and $m>n-d+1$. If $\h^{\bullet}(C_{n,d}^{\mathbb{Z}_\ell})$ 
	is torsion-free over $\mathbb{Z}_\ell$ then
	the cohomology of $C^R_{n,d}$ is given by the following formulas:
	\begin{equation}
	\h^{x}(C_{n,d}^{R})=\begin{cases}
	\nabla_R(\mu * x) \langle q^{x} \rangle, &  0\leq x<n-d \text{ and } x\leq d-1, \\
	\nabla_R(\mu * (x+1)) \langle q^{x+1}\rangle  &  n-d\leq x < d-1, \\
	\nabla_R(n) \langle q^n \rangle & x=2n-d-1, \\
	0 & \text{otherwise}.
	\end{cases}
	\end{equation}
\end{thm}	
We start with the following lemma:
\begin{lemma} \label{nomhook}
	Let $m$ be such that $m>d$.
	\begin{enumerate}
		\item If $x < n-d$ and $x \leq d-1$ then the partition $(n-d,x+1,1^{d-x-1})$ is an $m$-core unless $x=n-m$.
		\item If $n-d \leq x <d-1$ then the partition $(x,n-d+1,1^{d-x-1})$ is an $m$-core.
	\end{enumerate}	
\end{lemma}	
\begin{proof}
	In case (ii) the largest hook has length $x+1+(d-x-1)=d$. In case (i), the two largest hooks
	have length equal to $(n-d)+1+(d-x-1)=n-x$ and $x+1+d-x-1=d$ respectively. Note that $n-x=m$ can occur only if $m> n-d$.
\end{proof}	

\begin{proof}[\textsc{proof of the theorem}]
	If $m>n$ then $\ell \nmid |\Gl_n(q)|$, therefore the theorem follows easily from \ref{Xnd}. Therefore we shall now assume that $m \leq n$.
	
	\smallskip
	
	Let $\mu=(n-d)$ to be the trivial partition of $n-d$.  We cut the complex $C_{n,d}$ at different eigenvalues of $F$. Let $0 \leq x \leq n$. Assume $q^x \not\equiv q^n (\mathrm{mod}\hspace{0.1cm} \ell)$. By assumption on $m$ we have $2m=m+m>d+(n-d+1)> n$. Therefore, the condition on $x$ is equivalent to $x \neq n $ and $x \neq n-m$.
	In that case the cohomology of $C_{n,d}$  cut by the eigenvalue  $q^x$  can be non-zero in 5 at most one degree and the corresponding representation has character $\rho_{\mu *x}$. By \ref{nomhook}, $\mu*x$ is an $m$-core therefore  $\nabla_{R}(\mu *x)$ is the unique $R\Gl_n(q)$-module which lifts  to a $\mathbb{Q}_\ell \Gl_n(q)$-module with character $\rho_{\mu *x}$. 
\\
	Cutting the complex $C_{n,d}$ at the eigenvalue $q^n$ gives us two non-zero cohomology groups, one corresponding to the trivial representation with  eigenvalue $q^n$ and another one corresponding to the eigenvalue $q^{n-m}$ since $m \leq n$. Since $n-m<n-d$ this last cohomology group occurs in degree $n-m$ and it has character $\rho_{\mu *(n-m)}$. Let $M := \h^{n-m}(C_{n,d}^{\mathbb{Z}_\ell}) $. We note that the module $M \otimes_{\mathbb{Z}\ell} \mathbb{Q}_{\ell}$ is indecomposable and by assumption $M$ is torsion-free module over $\mathbb{Z}_{\ell}$, therefore $M$ is also indecomposable. Thus, it is indecomposable module with character $\rho_{\mu *(n-m)}$.
	Since the cohomology of $C_{n,d}^{\mathbb{Z}_\ell}$
	is torsion-free then $M \otimes_{\mathbb{Z}_\ell} \mathbb{F}_\ell \simeq \h(C_{n,d}^{\mathbb{F}_\ell})$. 
	Let $D $ be the generalized $q^n$-eigenspace of $F$ on $C_{n,d}^{\mathbb{F}_\ell}$. Then we have a distinguished triangle in $D^{b}(\mathbb{F}_\ell \Gl_n(q)\mmod)$ given by
	\[
	\xymatrix{
		M \otimes_{\mathbb{Z}_\ell} \mathbb{F}_\ell  [m-n] \ar[r] & D \ar[r] & S_{\mathbb{F}_\ell}(n)[d-2n+1] \ar@{~>}[r] &}
	\]
	in the stable category of $\mathbb{F}_\ell \Gl_n(q)$. 
	Since $M$ is indecomposable, we actually have $M \otimes_{\mathbb{Z}_\ell} \mathbb{F}_\ell \simeq \Omega^{n-d+m}  S_{\mathbb{F}_\ell}(n)$ as $\mathbb{F}_\ell \Gl_n(q)$-module. 
	 To conclude we only need to show that the socle of $M \otimes_{\mathbb{Z}_\ell} \mathbb{F}_\ell $ is isomorphic to  $S_{\mathbb{F}_\ell}(\mu *(n-m))$. For that purpose we use the Brauer tree (look at \cite[Section 5]{OLDL} for understanding how we can use the Brauer tree) of the $\Phi_m$-block, given by
	\begin{center}
		\begin{pspicture}(11,0.5)
		\psset{linewidth=1pt}
		\cnode[fillstyle=solid,fillcolor=black](0,0){5pt}{A}
		\cnode(1.5,0){5pt}{B}
		\cnode(3,0){5pt}{C}
		\cnode(6.5,0){5pt}{D}
		\cnode(8,0){5pt}{E}
		\cnode(10,0){5pt}{F}
		\cnode(11.5,0){5pt}{G}
		\ncline[nodesep=0pt]{A}{B}\ncput*[npos=0.5]{$S_m$}
		\ncline[nodesep=0pt,linestyle=dashed]{B}{C}\naput[npos=0]{$\vphantom{\big)}$}
		\ncline[nodesep=0pt]{C}{D}\ncput*[npos=0.5]{$S(\mu*(n-m))$}
		\ncline[nodesep=0pt]{D}{E}\ncput*[npos=0.5]{$T$}
		\ncline[nodesep=0pt,linestyle=dashed]{E}{F}\naput[npos=0]{$\vphantom{\big)}$}
		\ncline[nodesep=0pt]{F}{G}\ncput*[npos=0.5]{$S_1$}
		\end{pspicture}
	\end{center}
	\vspace{0.5cm}
	We label the vertices by $\chi_i$ for $i=1,\ldots,m$ and edges by $S_i$. 
	We can compute 
	 $\Omega^i k$ using the Brauer tree and then
	 we have $\Omega^{n-d+m}\mathbb{F}_\ell = \Omega^{n-d+m}S_1 $. 
	We first note that for $0\leq i < m$, $\Omega^i \mathbb{F}_\ell$  is isomorphic to
	$$ \Omega^{i}\mathbb{F}_\ell \simeq \begin{pmatrix}
	S_i \\
	S_{i+1}
	\end{pmatrix}$$
	which lifts to a $\mathbb{Z}_\ell G$-lattice with character $\chi_{i+1}$. Therefore
	$$ (\Omega^{i}\mathbb{F}_\ell )^{*} \simeq  \Omega^{-i}\mathbb{F}_\ell \simeq \Omega^{2m-i}\mathbb{F}_\ell  \simeq \begin{pmatrix}
	S_{i+1} \\
	S_{i}
	\end{pmatrix} $$
	which lifts to a $\mathbb{Z}_\ell G$-lattice with character $\chi_{i+1}$. Now, let $j \in \{1,\ldots,m\}$ be such that $\chi_j = \rho_{\mu *(n-m)}$, in which case $S_j \simeq S(\mu *(n-m))$ and $S_{j-1} = T$. We have
	$$ \Omega^{2m-j+1}\mathbb{F}_\ell \simeq \begin{pmatrix}
	T \\S(\mu *(n-m))
	\end{pmatrix}.$$
	Here the condition on $j$ becomes $m < 2m-j+1 \leq 2m$. On the other hand, by assumption on $m$ we have  $m < n-d+m < 2m$ (recall that $d< m \leq n$) and $\Omega^{n-d+m}\mathbb{F}_\ell$ lifts to a $\mathbb{Z}_\ell G$-lattice of character $\rho_{\mu *(n-m)}$. 
	Therefore we must have 
	$$\Omega^{n-d+m}\mathbb{F}_\ell \simeq \Omega^{2m-j+1}\mathbb{F}_\ell \simeq \begin{pmatrix}
	T \\S(\mu *(n-m))
	\end{pmatrix}$$
	which proves that $S(\mu *(n-m))$ is the socle of $M \otimes_{\mathbb{Z}_\ell} \mathbb{F}_\ell $ and shows that $M \simeq \nabla_{\mathbb{Z}_\ell}(\mu *(n-m))$.
	\smallskip
	
	Therefore this gives us the cohomology groups at each degree as follows
	\begin{equation}
	\h^{x}(C_{n,d}^{R})=\begin{cases}
	\nabla_{R}(\mu * x) \langle q^{x}\rangle &  x<n-d \text{ and } x\leq d-1, \\
	\nabla_{R}(\mu * (x+1)) \langle q^{x+1}\rangle &  n-d\leq x < d-1, \\
	\nabla_{R}(n) \langle q^n\rangle  & x=2n-d-1, \\
	0, & \text{otherwise}.
	\end{cases}
	\end{equation}
\end{proof}		 

\begin{eg}
	Let us consider the variety $\X_{4,3}$. When $\ell \mid \Phi_4(q)$ the cohomology groups of this variety cut by the eigenvalue $1$ has two non-zero terms in degree $4$ and $8$ respectively. We have the following distinguished triangle in $D^{b}(\mathbb{F}_\ell \Gl_4(q)\mmod)$
	\[
	\xymatrixcolsep{0.7cm}
	\xymatrixrowsep{0.7cm}
	\xymatrix{
		\h^{0} (C_{4,3}^{\mathbb{F}_\ell})[0]  \ar[r] & C \ar[r] &  S_{\mathbb{F}_\ell}(4)[-4] \ar@{~>}[r] & }.
	\]
	We have $\h^{0} (C_{4,3}^{\mathbb{F}_\ell})=\Omega^{5} S_{\mathbb{F}_\ell}(4)$. By walking around the Brauer tree of the principal $\ell$-block of $\Gl_4(q)$, we have have
	\begin{center}
		$\nabla_{\mathbb{F}_\ell}(1^4)=\Omega^5 S_{\mathbb{F}_\ell}(4) =  \begin{pmatrix}
		S_{\mathbb{F}_\ell}(21^2)\\
		S_{\mathbb{F}_\ell}(1^4)\\
		\end{pmatrix}
		\hookrightarrow
		\begin{pmatrix} 
		& S_{\mathbb{F}_\ell}(1^4) &        \\
		S_{\mathbb{F}_\ell}(1^4) &          &   S_{\mathbb{F}_\ell}(21^2)  \\
		. & & \\
		. & & \\
		S_{\mathbb{F}_\ell}(1^4)  & & \\
		&   S_{\mathbb{F}_\ell}(1^4)  &       \\
		\end{pmatrix}.$
	\end{center}	
	In addition, the $\mathbb{Z}_\ell G$-lattice $\nabla_{\mathbb{Z}_\ell}(1^4) \simeq \h^{4}(C_{4,3}^{\mathbb{Z}_\ell})$ has character $\rho_{(1^4)}$.
\end{eg}	
\begin{thm}\label{torsion-free} \label{Xnd2}
	Assume that $\ell \mid \Phi_m(q)$ with $m>d$ and $m>n-d+1$ and $m>6$. Then $\h^{i}_c(\X_{n,d},\mathbb{Z}_\ell)$ is torsion-free. 
\end{thm}
The following sections are devoted to proving the theorem by induction on $n$. More precisely,  we want to show that the theorem holds for $\X_{n,d}$ whenever it holds for $\X_{n-1,d}$ and $\X_{n-1,d-1}$. Note that the pair $(d,n-d)$ can only decrease therefore the assumption on $m$ carries over the induction.
\subsubsection{Base of induction}
Since $d \geq 1$ then for the base of induction, we should prove that the theorem holds for $\X_{1,1}$ and $\X_{n,1}$. 
\smallskip

The variety $\X_{1,1}$ is a point therefore the theorem holds trivially. For the second limit case we show the following proposition.
\begin{prop}
	Assume that $\ell \nmid |\Gl_n(q)|$. Then the cohomology of $\X_{n,1}$ over $\mathbb{Z}_{\ell}$ is torsion-free.
\end{prop}	
\begin{proof}
	Under the assumption on $\ell$ the category $\mathbb{Z}_\ell G\mmod$ is semisimple. The simple unipotent representations are exactly the representations $\nabla_{\mathbb{F}_\ell}(\lambda)$ where $\lambda $ runs over the partition of $n$. We use the notation of \cite{BDR}. Let $\mathcal{B}$ be the flag variety of $G$. There is a functor $\underline{\mathrm{Ind}}_F$ from the bounded $G$-equivariant derived category $\mathscr{D}$ of constructible $\mathbb{F}_\ell$-sheaves on $\mathcal{B} \times \mathcal{B}$ to $D^b(\mathbb{F}_\ell G\mmod)$. Under the assumption on $\ell$ it follows from \cite[Remark 3.9]{CDR} that $\underline{\mathrm{Ind}}_F$ and its right adjoint preserve the filtration by families. Therefore one can argue as in the proof of \cite[Theorem 4.2]{BDR} to show that given a partition $\lambda$ of $n$, 
	$$ \mathrm{Hom}_{\mathbb{F}_\ell G\mmod}\big(\h_c^i(\X(\mathbf{w_0}^2),\mathbb{F}_\ell), \nabla_{\mathbb{F}_\ell}(\lambda) \big) \simeq \mathrm{Hom}_{\mathbb{F}_\ell G\mmod}\big(\h_c^{i-n_\lambda}(\X(1),\mathbb{F}_\ell), \nabla_{\mathbb{F}_\ell}(\lambda) \big)$$
	where $n_\lambda = n(n-1) + 2\sum_i {\lambda_i\choose 2}$. Consequently the cohomology of $\X(\mathbf{w_0}^2)$ over $\mathbb{F}_\ell$ vanishes in odd degrees. Therefore by the universal coefficient theorem $\h_c^\bullet(\X(\mathbf{w_0}^2),\mathbb{Z}_\ell)$ is torsion-free. We conclude by remarking that  under our assumption on $\ell$ we have that $\h_c^\bullet(\X_{n,1},\mathbb{Z}_\ell)$ is a direct summand of $\h_c^\bullet(\X(\mathbf{w_0}^2),\mathbb{Z}_\ell)$ (see the proof of \cite[Corollary 3.2]{cohomology-of-DL_Dudas}).
\end{proof}	
\subsubsection{Inductive step}
Again we will work under assumption that $\ell \mid \Phi_m(q)$ with $m>d$ and $m>n-d+1$ and $m>6$. We will write $C_{n,d}$ instead of $C_{n,d}^{\mathbb{Z}_\ell}$.
By induction, we assume that $\h^{\bullet}(C_{n-1,d})$ and $\h^{\bullet}(C_{n-1,d-1})$ are torsion-free. Then we will prove that $\h^{\bullet}(C_{n,d})$ is torsion-free as well. If $m>n$ then the cohomology of $C_{n,d}$ admits the same description as in characteristic zero. Thus we can assume $m\leq n$. We will state several lemmas and theorems.
\begin{lemma} \label{TF1}
	If $^* \! \mathscr{R}^{\Gl_{n}}_{\Gl_{n-1}}(\h^{\bullet}(C_{n,d}))$  is torsion-free over $\mathbb{Z}_\ell$ so is $\h^{\bullet}(C_{n,d})$.
\end{lemma}	
\begin{proof}
	For convenience we work here with the cohomology of $\X_{n,d}$. Let $i \geq 0$ and  $T$ be the torsion part of $\h^{i}_c(\X_{n,d},\mathbb{Z}_\ell)$. Since $^* \! \mathscr{R}^{\Gl_{n}}_{\Gl_{n-1}}(\h^{i}_c(\X_{n,d},\mathbb{Z}_\ell))$ is torsion-free then $T$ is killed under Harish-Chandra restriction and therefore it is a cuspidal module (we note that by \cite{quantom} we have $^* \! \mathscr{R}^{\Gl_{n}}_{\Gl_{n-1}}(M) \neq 0 $  for any non-cuspidal  unipotent $\mathbb{Z}_{\ell}\Gl_n(q)$-module). By \ref{1,st1}, this forces $i =\dim(\X_{n,d})$ to be the middle degree. However by 
	explanation after \ref{boun1}
	 $\h^{\dim(\X_{n,d})}_c(\X_{n,d},\mathbb{Z}_\ell)$ is torsion-free.
\end{proof}	
Given $x \geq 0$,  by \ref{dm2} we have  some exact sequences
\[
\xymatrix{\sr \h^{x}(C_{n,d}^R) \ar[r] & \h^{x}(C_{n-1,d}^R) \ar[r]^{\hspace{-2cm}h} & \h^{x+1}(C_{n-1,d-1}^R) \oplus \h^{x}(C_{n-1,d-1}^R)(1) }
\]
for any ring among $\mathbb{Q}_\ell$, $\mathbb{Z}_\ell$ or $\mathbb{F}_\ell$. 
By assumption, $\h^{\bullet}(C_{n-1,d-1})$ and $\h^{\bullet}(C_{n-1,d})$ are torsion-free and using \ref{expXnd}, they are given by
\begin{equation}
\h^{x}(C_{n-1,d-1})=\begin{cases}
\nabla_{\mathbb{Z}_\ell}(\mu_1 * x) \langle q^{x}\rangle  &   x<n-d  \text{ and } x\leq d-2,\\
\nabla_{\mathbb{Z}_\ell}(\mu_1 * (x+1)) \langle q^{x+1}\rangle &  n-d\leq x < d-2, \\
\nabla_{\mathbb{Z}_\ell}(n-1) \langle q^{n-1}\rangle & x=2n-d-2,\\
0 & \text{otherwise},
\end{cases}
\end{equation}
and
\begin{equation}
\h^{x}(C_{n-1,d})=\begin{cases}
\nabla_{\mathbb{Z}_\ell}(\mu_2 * x) \langle q^{x}\rangle  &  x<n-d-1  \text{ and } x\leq d-1,\\
\nabla_{\mathbb{Z}_\ell}(\mu_2 * (x+1)) \langle q^{x+1}\rangle  &  n-d-1 \leq x < d-1, \\
\nabla_{\mathbb{Z}_\ell}(n-1) \langle q^{n-1}\rangle  & x=2n-d-3, \\
0 & \text{otherwise}.
\end{cases}
\end{equation}
where $\mu_1=(n-d)$ and $\mu_2=(n-d-1)$. We separate the problem in three cases with respect to the behavior of the boundary map $h$. In fact, we will see that the boundary map $h$,
\begin{center}
	$\h^{x}(C_{n-1,d}) \longrightarrow^{\hspace{-0.4cm} h} \hspace{0.1cm} \h^{x+1}(C_{n-1,d-1}) \oplus \h^{x}(C_{n-1,d-1})(1)$
\end{center}	
is zero in all degrees except in degrees $x=n-d-1$ (which occurs only if $n+1<2d$) and $x=2n-d-3$. For this we shall use the eigenvalues of $F$, which are preserved by $h$. In the case where $h$ is non-zero we will obtain  exact sequences of length $4$. 
\bigskip

\noindent
$\bullet$ \textbf{Step I}
\smallskip

\noindent
- For $0 \leq x< n-d-1$ and $x\leq d-2$, the boundary map $h$ is given by
\[
\xymatrix{ \nabla_{\mathbb{Z}_\ell}(\mu_{2} * x) \langle q^x\rangle \ar[r]^{\hspace{-2.5cm} h} & \nabla_{\mathbb{Z}_\ell}(\mu_1 * (x+1)) \langle q^{x+1}\rangle \oplus \nabla_{\mathbb{Z}_\ell}(\mu_1 * x) \langle q^{x+1}\rangle}
\]
is a zero map since $q \neq 1$ in $\mathbb{F}_\ell$ (since $m >d \geq 1$).
This yields short exact sequences of the form:
$$
\xymatrix{  0 \ar[r] & \mbox{$ \begin{array}{c} \h^{x}(C_{n-1,d-1}) \\ \oplus \\  \h^{x-1}(C_{n-1,d-1})(1) \end{array}$} \ar[r] & \sr(\h^{x}(C_{n,d})) \ar[r] & \h^{x}(C_{n-1,d}) \ar[r]^{\hspace{0.7cm} h} & 0}.
$$
Since the left and right terms of the short exact sequences are torsion-free by assumption, the middle terms will be also torsion-free over $\mathbb{Z}_\ell$.  By \ref{TF1}, $\h^x(C_{n,d})$ is also torsion-free in these cases. 
\smallskip

\noindent
- For $n-d\leq x< d-2$, the boundary map $h$ is given by
\begin{small}
	\[
	\xymatrix{ \nabla_{\mathbb{Z}_\ell}(\mu_{2} * (x+1)) \langle q^{x+1}\rangle  \ar[r] & \nabla_{\mathbb{Z}_\ell}(\mu_1 * (x+2)) \langle q^{x+2} \rangle \oplus \nabla_{\mathbb{Z}_\ell}(\mu_1 * (x+1)) \langle q^{x+2}\rangle}
	\]
\end{small}
Hence, we can conclude as before.
\smallskip

\noindent
-For $x=d-2$ and $x\geq n-d-1$, since the right side of the map is zero then the boundary map $h$ is given by
\[
\xymatrix{\nabla_{\mathbb{Z}_\ell}(\mu_2 *(x+1)) \langle q^{x+1}\rangle \ar[r] & 0   }
\]

\bigskip

\noindent
$\bullet$ \textbf{Step II}
\smallskip

\noindent  
For $x=n-d-1$ and $x \leq d-2$, the boundary map $h$ can be non-zero. We will start by considering the version of this map over $\mathbb{F}_\ell$.
It is given by
\[
\xymatrixcolsep{0.3cm}
\xymatrix{\nabla_{\mathbb{F}_\ell}(\mu_2 * (x+1)) \langle q^{x+1}\rangle \ar[r]^{\hspace{-2cm}\bar h} & \nabla_{\mathbb{F}_\ell}(\mu_1 * (x+2)) \langle q^{x+2} \rangle   \oplus \nabla_{\mathbb{F}_\ell} (\mu_1 * (x)) \langle q^{x+1}\rangle }.
\]
If we cut the sequence by the eigenvalue $q^{x+1}=q^{n-d}$, the map $h$ is given by:
\[
\xymatrix{\nabla_{\mathbb{F}_\ell}(\mu_2 * (x+1)) \langle q^{x+1}\rangle  \ar[r]^{\hspace{0,5cm}\bar h} & \nabla_{\mathbb{F}_\ell} (\mu_1 * (x)) \langle q^{x+1}\rangle  }
\]
The partition $\mu_2 *(n-d)=(n-d-1) * (n-d)$ has 
\begin{center}
	$ \{n-1,d-1\cdots,2,1,0 \} \setminus \{n-d \} \cup \{n\}=\{n,n-1,d-1,\cdots,2,1,0 \}  \setminus \{n-d \}$ 
\end{center}
as a $\beta$-set.
It corresponds to the partition $(n-d,n-d,1^{2d-n-1})$ of $n-1$. We can also see that the partition $(n-d) * (n-d-1)$ is $(n-d,n-d,1^{2d-n-1})$. Consequently, we obtain an exact sequence of the form
\begin{small} 
	\[
	\xymatrixcolsep{0.6cm}
	\xymatrix{ 0 \ar[r] & \sr (\h^{n-d-1}(C_{n,d}^{\mathbb{F}_\ell})_{q^{n-d}}) \ar[r] & \nabla_{\mathbb{F}_\ell}(n-d,n-d,1^{2d-n-1}) \langle q^{n-d}\rangle    &   }
	\]
	\[
	\xymatrixcolsep{0.6cm}
	\xymatrix{
		\ar[r]^{\hspace{-2.6cm}\bar h} & \nabla_{\mathbb{F}_\ell}(n-d,n-d,1^{2d-n-1}) \langle q^{n-d}\rangle \ar[r] &\sr(\h^{n-d} (C_{n,d}^{\mathbb{F}_\ell})_{q^{n-d}}) \ar[r] & 0}
	\]
\end{small}
We will prove that $\bar h$ is an isomorphism which will show that  $\h^{n-d-1}(C_{n,d}^{\mathbb{F}_\ell})_{ q^{n-d}}$ and $\h^{n-d}(C_{n,d}^{\mathbb{F}_\ell})_{ q^{n-d}}$ are zero.
For that purpose we use the following lemma.
\begin{lemma}\label{simproj}
	Assume $2d \geq n+1$ and $(n,d) \neq (6,4)$, let $\gamma = (n-d,n-d,1^{2d-n-1})$. Assume furthermore that  $\ell \mid \Phi_m(q)$ with $m >d$ and $m >n-d+1$. There is no complex $D$ of  $\mathbb{F}_\ell \Gl_n(q)$-modules such that
	\begin{itemize}
		\item $D$ is a perfect complex.
		\item For all $i \neq 0,1$ we have $\h^{i}(D)=0$.
		\item $\sr \h^{0}(D)=  \sr\h^{1}(D)=S_{\mathbb{F}_\ell}(\gamma)$. 
		\item $\h^{1}(D)$ has no cuspidal composition factor.
	\end{itemize}
\end{lemma}
\begin{proof}	
	We first note that the largest hook of $\gamma$ has length $d$ therefore $S_{\mathbb{F}_\ell}(\gamma)$ is simple and projective. We then observe that if the Harish-Chandra restriction of $\h^{1}(D)$ is a simple module then $S:=\h^{1}(D)$  is a simple module. Indeed, if we can write  $S$ as 
	\[
	\xymatrix{ 0 \ar[r] & S_1 \ar[r] & S \ar[r] & S_2 \ar[r] & 0 }
	\]
	then we have  
	\[
	\xymatrix{0 \ar[r] & \sr(S_1) \ar[r] & \sr(S) \ar[r] & \sr(S_2) \ar[r] & 0}.
	\]
	Since $\sr(S)$ is a simple module then $\sr(S_1)=0$ or $\sr(S_2)=0$. We deduce that $S_1$ or $S_2$ is cuspidal module. By assumption,  one of the two must be zero, therefore $S$ is simple. 
	\smallskip

	Since  $D$ is perfect complex then by \ref{omegarick}, we have
	\begin{align} \label{sakht}
	\Omega^2 \h^{1}(D)\simeq \Omega^2 S \simeq 
	\h^0(D)
	\end{align}
	in the stable category of $\mathbb{F}_\ell \Gl_n(q)$. Let $\lambda$ be a unique partition of $n$ such that
	$S=S_{\mathbb{F}_\ell}(\lambda)$.
	We deduce that
	\begin{align} 
	\Omega^2 S_{\mathbb{F}_\ell}(\lambda) \simeq L
	\end{align}
	where $L$ is simple modulo cuspidal composition factors.
	For the simplicity  we now remove the subscript $\mathbb{F}_\ell$ from the modules.
	\begin{itemize}
		\item  If $S(\lambda) $
		is a projective module then
		\begin{center}
			$S(\lambda)\simeq \nabla(\lambda)$.
		\end{center}	
		We can look at the image of the Harish-Chandra restriction of $\nabla(\lambda)$ in the Grothendieck group:
		$$[\sr(\nabla(\lambda))]=\Sigma_{\lambda \setminus \mu=\square}[\nabla(\mu)]=[\nabla(\gamma)].$$
		Consequently,  $\gamma$ is the unique partition of $n$ which is obtained  from $\lambda$ by removing a box from the Young diagram. The only possibilities for $\lambda$ are the partitions $(n-d+1,n-d,1^{2d-n-1})$, $(n-d,n-d,2,1^{2d-n-2})$ and $(n-d,n-d,1^{2d-n})$ But, none of them is acceptable since we can remove several boxes from these partitions.
		\item If 
		$S=S(\lambda)$  is non-projective. Then it lies in a block with non-trivial cyclic defect group. We will prove that $\Omega^2 S(\lambda)$ is never a simple module. Assume first that  $S(\lambda)$ is a leaf in the Brauer tree
		of the block of the following form
		\begin{center}
			\begin{pspicture}(7.7,0.5)
			\psset{linewidth=1pt}
			\cnode(0,0){5pt}{A}
			\cnode(3,0){5pt}{B}
			\cnode(5,0){5pt}{C}
			\cnode(7,0){5pt}{D}
			\ncline[nodesep=0pt,linestyle=dashed]{A}{B}\naput[npos=0]{$\vphantom{\big)}$}
			\ncline[nodesep=0pt]{B}{C}\ncput*[npos=0.5]{$T$}
			\ncline[nodesep=0pt]{C}{D}\ncput*[npos=0.5]{$S$}
			\end{pspicture}
		\end{center}
		Then $S=\nabla(\lambda)$. By the same argument as above there is no $\lambda$ such that $\sr(\nabla(\lambda))=\nabla(n-d,n-d,1^{2d-n-1})$. Therefore $S$ can not be a leaf and the Brauer tree is of the following shape.
		\begin{center}
			\begin{pspicture}(12,0.5)
			\psset{linewidth=1pt}
			\cnode(0,0){5pt}{A}
			\cnode(2,0){5pt}{B}
			\cnode(4,0){5pt}{C}
			\cnode(6,0){5pt}{D}
			\cnode(8,0){5pt}{E}
			\cnode(10,0){5pt}{F}
			\cnode(12,0){5pt}{G}
			\ncline[nodesep=0pt,linestyle=dashed]{A}{B}\naput[npos=0]{$\vphantom{\big)}$}
			\ncline[nodesep=0pt]{B}{C}\ncput*[npos=0.5]{$U_1$}
			\ncline[nodesep=0pt]{C}{D}\ncput*[npos=0.5]{$U$}
			\ncline[nodesep=0pt]{D}{E}\ncput*[npos=0.5]{$S$}
			\ncline[nodesep=0pt]{E}{F}\ncput*[npos=0.5]{$T$}
			\ncline[nodesep=0pt]{F}{G}\ncput*[npos=0.5]{$T_1$}
			\end{pspicture}
		\end{center}
		By walking around  the Brauer tree, we have
		\begin{center}
			$ \begin{pmatrix} 
			&  U & \\
			U_1 &  & S \\
			& U & \\
			\end{pmatrix} \bigoplus 
			\begin{pmatrix} 
			&  T & \\
			S &  & T_1 \\
			& T & \\
			\end{pmatrix}
			= P_{U} \bigoplus P_{T} 
			\longrightarrow \begin{pmatrix} 
			& S &        \\
			U &          &   T  \\
			&   S  &       \\
			\end{pmatrix}  \longrightarrow^{\hspace{-0.3cm}f} S$
		\end{center}
		and therefore
		\begin{center}
			$\Omega^{2} S \simeq  \begin{pmatrix} 
			U_1 &  & S &  & T_1  \\ 
			& U &  & T & &    \\
			\end{pmatrix} 
			$
		\end{center}
		with $U$ and $T$ being non-zero. Thus in these cases $\Omega^2 S$ can not be a simple module modulo cuspidal composition factors.  
		We note that our argument does not concern the case where Brauer tree is of the following form (in other words, when $m=1$.)
		
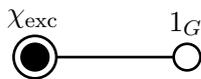
\begin{center}
		\begin{pspicture}(2,0.5)
		\psset{linewidth=1pt}
			\cnode[fillstyle=solid,fillcolor=black](0,0){5pt}{A}
		\cnode(0,0){8pt}{A}
			\cnode(2,0){5pt}{B}
			\ncline[nodesep=0pt]{A}{B}\naput[npos=1.1]{$\vphantom{\big)}1_G$}\naput[npos=-0.2]{$\vphantom{\Big)}\chi_\text{exc}$}
			\end{pspicture}
		\end{center}
		 This case does not happen in this lemma since we assume $m>d \geq 1$. Therefore, the cases $\lambda = (1^2)$ or $\lambda = (2)$ are not considered in this lemma.
	\end{itemize}
\end{proof}	
We can apply the previous lemma  to the complex $D:=(C_{n,d}^{\mathbb{F}_\ell})_{q^{n-d}}[x]$ because $n \geq m >6$. Hence the map $\bar h$  
\[
\xymatrix{\nabla_{\mathbb{F}_\ell}(n-d,n-d,1^{2d-n-1}) \langle q^{x+1}\rangle  \ar[r]^{\hspace{0.2cm}\bar h} & \nabla_{\mathbb{F}_\ell} (n-d,n-d,1^{2d-n-1}) \langle q^{x+1}\rangle }
\]
can not be zero. Therefore it is an isomorphism since we saw that $\nabla_{\mathbb{F}_\ell} (n-d,n-d,1^{2d-n-1}) $ is simple and projective. Consequently $\h^{n-d-1}(C_{n,d}^{\mathbb{F}_\ell})_{ q^{n-d}}$ and $\h^{n-d}(C_{n,d}^{\mathbb{F}_\ell})_{ q^{n-d}}$  are zero. By the universal coefficient theorem 
$\h^{n-d-1}(C_{n,d})_{ q^{n-d}}$ and $\h^{n-d}(C_{n,d})_{ q^{n-d}}$ are also zero hence torsion-free. 
Note that the case $n+1=2d$ does not occur in this step.
\bigskip

\noindent
$\bullet$
\textbf{Step III}
\smallskip

\noindent
Let $x=2n-d-3$. As in \textbf{Step II} we work over $\mathbb{F}_\ell$  and we consider the boundary map $\bar g$,  cut by the eigenvalue $q^{x+1}$  which is given by 
\begin{small}
	\[
	\xymatrixcolsep{0.5cm}
	\xymatrixrowsep{0.2cm}
	\xymatrix{ 0 \ar[r] &  \sr(\h^{2n-d-3}(C_{n,d}^{\mathbb{F}_{\ell}})) \ar[r] & \nabla_{\mathbb{F}_\ell}(n-1) \langle q^{n-1}\rangle \ar[r]^{\hspace{0,2cm}g} &  \nabla_{\mathbb{F}_\ell}(n-1) \langle q^{n-1}\rangle   & }
	\]
	\[
	\xymatrixcolsep{0.5cm}
	\xymatrixrowsep{0.8cm}
	\xymatrix{   \hspace*{5cm} \ar[r] &  \sr(\h^{2n-d-2}(C_{n,d}{\mathbb{F}_{\ell}}))
		\ar[r] & 0}.
	\]	
\end{small}	
\noindent
We shall prove that $\bar g$ is an isomorphism. It will follow  from universal coefficient theorem that  $\h^{2n-d-2}(C_{n,d})$ and $\h^{2n-d-1}_c(C_{n,d})$ are zero hence torsion-free. 
\smallskip

Let us assume by contradiction that $\bar g$ is not an isomorphism. Then it must be zero. Consequently 
$$ \sr(\h^{2n-d-3}(C_{n,d}^{\mathbb{F}_\ell})_{q^{n-1}}) \simeq \sr(\h^{2n-d-2}(C_{n,d}^{\mathbb{F}_\ell})_{q^{n-1}}) \simeq \nabla_{\mathbb{F}_\ell}(n-1) \simeq  \mathbb{F}_\ell.$$
\begin{lemma}
	Assume that  $\ell \mid \Phi_m(q)$ with $m >d$ and $m >n-d+1$. Let $a>0$ and  $D$ be a complex of  $\mathbb{F}_\ell \Gl_n(q)$-modules such that
	\begin{itemize}
		\item $D$ is a perfect complex.
		\item For all $i \neq  0,a,a+1$ we have $\h^{i}(D)=0$.
		\item $\h^{a}(D)= \h^{a+1}(D)=\mathbb{F}_\ell$. 
	\end{itemize}
	Then 
	$$
	\h^{0}(D)[0] \simeq \mathbb{F}_\ell[-a-1] \oplus \mathbb{F}_\ell[-a-2]
	$$
	in $\mathbb{F}_\ell \Gl_n(q)\stab$. 
\end{lemma}	
\begin{proof}
	We consider the distinguished triangle 
	\[
	\xymatrix{
		\tau_{\leq 0} D \ar[r] & D \ar[r] & \tau_{>0} D \ar@{~>}[r] &}
	\]
	in $D^{b}(\mathbb{F}_\ell G\mmod)$.
	The complex $\tau_{\leq 0} D$ is concentrated in one degree and this term is isomorphic to $\h^{0}(D)$. Since $\tau_{>0} D$,
has two non-zero cohomology degrees, hence we have a distinguished triangle
	\[
	\xymatrix{ \h^{a}(D)[-a] \ar[r] & \tau_{>0} D \ar[r] & \h^{a+1}(D)[-a-1] \ar@{~>}[r] & }
	\] 
	which can be written as 
	\[
	\xymatrix{\mathbb{F}_\ell[-a] \ar[r] & \tau_{>0} D \ar[r] & \mathbb{F}_\ell[-a-1] \ar@{~>}[r] & }.
	\] 
	By shifting this distinguished triangle we obtain
	\[
	\xymatrix{\mathbb{F}_\ell[-a-2] \ar[r]^{\hspace{0.3cm}f} &  \mathbb{F}_\ell[-a] \ar[r] & \tau_{>0} D  \ar@{~>}[r] & }
	\]
	in $D^{b}(\mathbb{F}_lG\mmod)$. Since $\ell \mid \Phi_m(q)$ and using \cite{Quillen} we can deduce that
	\begin{align*}
	\Hom_{D^{b}(  \mathbb{F}_lG\mmod)}(\mathbb{F}_\ell[-a-2],\mathbb{F}_\ell[-a])  & = \Ext^2_{\mathbb{F}_lG}(\mathbb{F}_\ell,\mathbb{F}_\ell) \\ &=\h^{2}(\Gl_n(q)) \\ &=
	0. 
	\end{align*}	
	Thus $f$ is a zero map in $D^{b}(\mathbb{F}_lG \mmod)$. It gives us 
	\begin{center}
		$ \mathbb{F}_\ell[-a-2+1] \oplus \mathbb{F}_\ell[-a] \simeq \tau_{>0}D$.
	\end{center}	
	Ultimately, we obtain the distinguished triangle in $D^{b}(\mathbb{F}_\ell G\mmod)$
	\[
	\xymatrix{ \h^{0}(D)[0] \ar[r] & D \ar[r] & \mathbb{F}_\ell[-a-1] \oplus \mathbb{F}_\ell [-a]  \ar@{~>}[r] & }.
	\]
	Now, we consider  the image of the above distinguished triangle in the stable category of $\mathbb{F}_\ell G$. By \ref{perfect_complex} we deduce 
	\begin{equation} \label{eq:step3}
	\h^{0}(D)[0] \simeq \mathbb{F}_\ell[-a-1] \oplus \mathbb{F}_\ell[-a-2]
	\end{equation}
	in $\mathbb{F}_\ell G\stab$. 
\end{proof}
Let us consider the generalized $q^{n-1}$-eigenspace of $F$ on $C_{n,d}^{\mathbb{F}_\ell}$. It has two consecutive cohomology groups isomorphic to $\mathbb{F}_\ell$. Furthermore, it follows from  \textbf{Step I} that it has an additional non-zero cohomology group in a smaller degree, which corresponds to the $\ell$-reduction of  the  eigenvalue $q^{n-1-m}$. This cohomology group is isomorphic to $\nabla_{\mathbb{F}_\ell}(n-d,n-m,1^{m+d-n})$ which is simple and projective with the assumption on $m$ (see \ref{nomhook}). Therefore we can apply \eqref{eq:step3} to get 
$$
\nabla_{\mathbb{F}_\ell}(n-d,n-m,1^{m+d-n})[0] \simeq \mathbb{F}_\ell[-a-1] \oplus \mathbb{F}_\ell[-a-2]
$$
in $\mathbb{F}_\ell G\stab$ for some $a>0$. This contradicts the fact that $\nabla_{\mathbb{F}_\ell}(n-d,n-m,1^{m+d-n})$ is projective hence zero in stable category.
This concludes the proof of the theorem \ref{torsion-free}.

\section{Cohomology complex of $\X_{n,d}$}
\subsection{Partial-tilting complex}
Throughput this section, we assume $k=\mathbb{F}_\ell$ and $K=\mathbb{Q}_\ell$. 
\begin{dfn}\label{defpartial}
	Let $R$ be a field and $A$ be an $R$-algebra. A complex $C \in D^b(A\mmod)$ is called a partial-tilting complex if
	\begin{itemize}
		\item $C$ is a perfect complex,
		\item $\Hom_{D^b(A\mmod)}(C,C[a]) = 0$ for all $a\neq 0$ (self-orthogonality condition).
	\end{itemize}
\end{dfn}	

\subsection{Self-orthogonality condition for cohomology complex of $\X_{n,d}$}
Our aim of this section is proving the following theorem:
\begin{thm} \label{result-partial_tilting}
	Assume $\ell \mid \Phi_m(q)$ with $m>n-d+1$ and $m>d$ and $m>6$ then $\R\Gamma_{c}(\X_{n,d},k)$ is a partial-tilting complex.
\end{thm}	

First note that the theorem holds trivially if $m > n$, in which case $k\Gl_n(q)$ is semisimple. Therefore from now on we shall always assume $m \leq n$. Since $m> d$ this will imply $d < n$, therefore we will not consider the case of the Coxeter variety.

\smallskip

Recall from \ref{Xnd} that the eigenvalues of $F$ on $\h^{\bullet}_c(\X_{n,d},k)$ are of the form  $q^{n-d-1}$, $q^{n-d}$,$\ldots$, $q^{n-2}$ and $q^{2n-d-1}$ (note that the eigenvalue $q^{2n-2d-1}$ does not occur). Given $i \in \{0,\ldots, m-1\}$ we denote by $C_i:=\R \Gamma_{c}(\X_{n,d},k)_{q^{i}}$ the generalized $q^i$-eigenspace of $F$ on the cohomology complex of $\X_{n,d}$. 
 Then we have $$\R \Gamma_{c}(\X_{n,d},k) \simeq \bigoplus_{i=0}^{m-1} C_{i}. $$
Note that each $C_i$ is a perfect complex.
\smallskip

By \ref{torsion-free}, if $i \not\equiv 2n-d-1 (\mathrm{mod}\ m)$, then the cohomology of $C_i$ is concentrated in one degree and it isomorphic to $\nabla_{k}(\lambda_i)$ for some partition $\lambda_i $ of $n$. It follows from \ref{nomhook} that $\nabla_{k}(\lambda_i)$  is simple and projective. Therefore $C_i \simeq \nabla_{k}(\lambda_i)[d_i]$ in $K^b(kG\mproj)$.  We deduce that for all  such $i \neq j$ and all integer $a$ we have $$ 	\Hom_{D^b(kG\mmod)}(C_i,C_j[a])=\Ext_{kG}^{a+d_j-d_i}(\nabla_k(\lambda_i),\nabla_k(\lambda_j))=0$$
since $\nabla_{k}(\lambda_i)$ and  $\nabla_k(\lambda_j)$ are non-isomorphic projective simple modules.
\smallskip

The direct summand  corresponding to $q^{2n-d-1}$ is more complicated. We denote it by $D:=\R\Gamma_{c}(\X_{n,d},k)_{q^{2n-d-1}}$. It has two non-zero cohomology groups, one in top degree with the trivial representation $\nabla_k(n)$ and another which corresponds to the eigenvalue  $q^{2n-d-1-m}$ of $F$ on $\h^{\bullet}_c(\X_{n,d},K)$. It is the direct summand of $\rgc(\X_{n,d},k)$ corresponding to the principal block. Therefore for all $a$ and all $i \not\equiv 2n-d-1 (\mathrm{mod}\ m)$ we have 
$$\Hom_{D^b(A\mmod)}(C_i,D[a])=\Hom_{D^b(A\mmod)}(D,C_i[a])=0.$$
Consequently theorem \ref{result-partial_tilting} holds if and only if $D$ is a partial-tilting complex. We are now to study this 
complex in details.
\begin{prop} \label{Rickard3}
	Assume $\ell \mid \Phi_m(q)$ where $n \geq m$,  $m>n-d +1$ and $m>d$ and $m>6$. Write the Brauer tree of the principal block of $k \Gl_n(q)$  as follows 
	\begin{center}
		\begin{pspicture}(13,0.5)
		\psset{linewidth=1pt}
		\cnode[fillstyle=solid,fillcolor=black](0,0){5pt}{A}
		\cnode(1.7,0){5pt}{B}
		\cnode(3.7,0){5pt}{C}
		\cnode(5.7,0){5pt}{D}
		\cnode(7.4,0){5pt}{E}
		\cnode(9.4,0){5pt}{F}
		\cnode(11.1,0){5pt}{G}
		\cnode(12.8,0){5pt}{H}
		\ncline[nodesep=0pt]{A}{B}\ncput*[npos=0.5]{$m$}
		\ncline[nodesep=0pt]{B}{C}\ncput*[npos=0.5]{$m-1$}
		\ncline[nodesep=0pt,linestyle=dashed]{C}{D}\naput[npos=0]{$\vphantom{\big)}$}
		\ncline[nodesep=0pt]{D}{E}\ncput*[npos=0.5]{$j$}
		\ncline[nodesep=0pt,linestyle=dashed]{E}{F}\naput[npos=0]{$\vphantom{\big)}$}
		\ncline[nodesep=0pt]{F}{G}\ncput*[npos=0.5]{$2$}
		\ncline[nodesep=0pt]{G}{H}\ncput*[npos=0.5]{$1$}
		\end{pspicture}
	\end{center}
	\vspace{0.5cm}
	where the trivial representation $\nabla_K(n)$ is the rightmost vertex in the tree.
	Then $D:=\rgc(\X_{n,d},k)_{q^{2n-d-1}}$ is isomorphic to a complex of the form
	\[
	\xymatrixcolsep{0.6cm}
	\xymatrixrowsep{0.5cm}
	\xymatrix{ 0 \ar[r] & P_{j} \ar[r] & P_{j+1} \ar[r] & P_{j+2} \ar[r] & \cdots \ar[r] & P_{m} \ar[r]& P_{m} \ar[r] & \cdots \ar[r] & P_{2} \ar[r]&  P_{1} \ar[r] & 0 }
	\]
	for some $2 \leq j \leq m$.
\end{prop}	
\begin{proof}
	Since the Brauer tree of the block of $\Gl_n(q)$ is a straight line then there is a unique path  from the  vertex $\nabla_K(n)$  to any other vertices. By walking around the Brauer tree from that vertex, one can construct a minimal projective resolution of the trivial module $k=\nabla_k(n)$, and truncate it to obtain the following complex  
	\[
	\xymatrixcolsep{0.6cm}
	\xymatrixrowsep{0.5cm}
	E:= \xymatrix{ 0 \ar[r] & P_{j} \ar[r] & P_{j+1} \ar[r] & P_{j+2} \ar[r] & \cdots \ar[r] & P_{m} \ar[r]& P_{m} \ar[r] & \cdots \ar[r] & P_{2} \ar[r]&  P_{1} \ar[r] & 0 }
	\]
	for any $1 \leq j \leq m$.
	The cohomology of this complex equals $k$ in the top degree and $\Omega^{-j}k$ in the bottom degree, therefore we have a distinguished triangle
	\begin{equation}\label{eq:E}
	\xymatrix{ E \ar[r] &  k[0] \ar[r] & \Omega^{-j}k [2m-j] \ar@{~>}[r] &  }.
	\end{equation}
	in $D^b(kG\mmod)$.
	\smallskip

	We have already seen that $D$ is a perfect complex. By the explanation before \ref{Rickard3}  the cohomology of $D$  vanishes outside two degrees. One is the top degree, equal to $\alpha := 4n-2d-2$, where the trivial representation occurs. The other one is the degree in which the eigenvalue $q^{2n-d-1-m}$ occurs. Writing $2n-d-1-m$ as $(n-d-1)+(n-m)$ we see that this degree equals to $\beta:=2n-d-1+m-n=3n-d-1-m$ and that the corresponding cohomology group is 
	$$  \h^\beta(D) \simeq \nabla_k(\mu*(n-m))=\nabla_k(n-d,n-m+1,1^{d-n+m-1}).$$
	We have a distinguished triangle 
	\[
	\xymatrix{ D \ar[r] &  k[- \alpha] \ar[r]^{\hspace{-0.6cm}f} & \h^{\beta}(D)[-\beta +1] \ar@{~>}[r] &  }.
	\] 
	Then $f \in \Hom_{D^b(kG\mmod)}(k[- \alpha],\h^{\beta}(D)[-\beta +1])$. By  the explanation in the proof of \ref{expXnd}, we must have $\h^{\beta}(D) \simeq \Omega^{\alpha - \beta +1} k$ since $D$ is a perfect complex.  
	We have
	$$\begin{aligned}
	\Hom_{D^b(kG\mmod)}(k[- \alpha],\h^{\beta}(D)[-\beta +1]) & \, \simeq \Hom_{D^b(kG\mmod)}(k[- \alpha],\Omega^{\alpha - \beta +1} k[-\beta +1])
	\\ &\, \simeq \Ext^{\alpha - \beta +1}_{kG}(k,\Omega^{\alpha - \beta +1}k) \\
	&\, \simeq \Ext^{0}_{kG}(k,k).
	\end{aligned}	$$
	Since $\Ext^{0}_{kG}(k,k) \simeq \Hom_{kG\mmod}(k,k) \simeq k$, we deduce that $f$ is unique up to a scalar. Now $\alpha-\beta+1 = n-d+m$. If we set $j := m-n+d$ then  $\alpha-\beta+1  = 2m-j$ and it follows from \eqref{eq:E} and the previous discussion that $D$ is isomorphic to $E[-\alpha]$. Note that $j>1$ since $ m > n-d +1$.
\end{proof}
\begin{rem}
The hypothesis $m>6$ can be relaxed in theorem \ref{torsion-free} for specific values of $n$ and $d$. The only restriction is to avoid hitting $n=6$ and $d=4$ in the induction process, because lemma \ref{simproj}  does not work in this case. For instance, if $m=5$ and  $(n,d) = (5,4)$ the cohomology is torsion free. We will use it in the next example. 
\end{rem}
\begin{eg}
	Assume  $\ell \mid \Phi_5(q)$. For the variety $\X_{5,4}$. From \ref{expXnd} we can  obtain the following table
	$$
	\begin{array}{c|ccccccc} i   & 5 & 6 & 7 & 8&9 &10
	\\\hline \\[0.08cm]
	\h^i_c(\X_{5,4},k) & \begin{pmatrix}  S_{k}(21^3) \\ S_{k}(1^5)   \end{pmatrix} \langle 1 \rangle & S_{k}(2^21)\langle q^2\rangle & S_{k} (32) \langle q^3 \rangle & 0  & 0 & S_{k}(5) \langle q^5\rangle \\[0.3cm]
	\end{array}
	$$
	For simplicity, we remove the subscript  $k$ from the modules. The Brauer tree of the principal block of $\Gl_5(q)$ is a straight line and  the cohomology complex, cut by the eigenvalue $1 \equiv q^5$, is given by 
	\[
	\xymatrixcolsep{0.7cm}
	\xymatrixrowsep{0.7cm}
	D \simeq	\xymatrix{ 0 \ar[r] & P(1^5)  \ar[r] & P(1^5) \ar[r] & P(21^3) \ar[r] & P(31^2) \ar[r] & P(41) \ar[r] & P(5) \ar[r] & 0 }
	\]
	with
	\begin{center}
		\begin{small}
			$P(1^5):= \begin{pmatrix}
			& S(1^5) &        \\
			S(1^5) &          &   S(21^3)  \\
			S(1^5) &          &       \\
			S(1^5) &         &    \\
			.          &         &     \\
			.         &         &     \\
			.       &          &     \\
			& S(1^5) & \\
			\end{pmatrix}$,
			$P(21^3):=\begin{pmatrix} 
			& S(21^3) &        \\
			S(1^5) &          &   S(31^2)  \\
			&   S(21^3)  &       \\
			\end{pmatrix}$, 
			\vspace{0.8cm}
			
			$P(31^2):=\begin{pmatrix} 
			& S(31^2) &        \\
			S(21^3) &          &   S(41)  \\
			&   S(31^2)  &       \\
			\end{pmatrix}$,
			$P(41):=\begin{pmatrix} 
			& S(41) &        \\
			S(31^2) &          &   S(1^5)  \\
			&   S(41)  &       \\
			\end{pmatrix}$,
			$P(5):=\begin{pmatrix} 
			S(5)        \\
			S(41)  \\
			S(5)      \\
			\end{pmatrix}.$
		\end{small}
	\end{center}
	For the eigenvalues $q^2$ and $q^3$ the cohomology complexes are concentrated in one degree. Those degrees are simple and projective modules of the form
	\begin{align*}
	&&	C_2:&  0 \longrightarrow 0  \longrightarrow \cdots \longrightarrow S(2^21)  \longrightarrow 0 \longrightarrow  0, && \\ \text{and}
	&& 
	C_3: & 0 \longrightarrow 0  \longrightarrow 0 \longrightarrow \cdots   \longrightarrow  S(32) \longrightarrow 0. 
	\end{align*} 
	Therefore the cohomology complex of $\X_{5,4}$ is 
	$$\rgc(\X_{n,d},k) \, \simeq  \displaystyle D \oplus C_2 \oplus C_3 $$
	which is given by
	\[
	\xymatrixcolsep{0.4cm}
	\xymatrixrowsep{0.5cm}
	\xymatrix{ 0  \ar[r] & P(1^5)  \ar[r] & P(1^5) \oplus P(2^21) \ar[r] & P(21^3)\oplus P(32) \ar[r] & P(31^2) \ar[r] & P(41) \ar[r] & P(5) \ar[r] & 0}.
	\]
\end{eg}	

In the next lemma, we will show that for different integers $a$, we have three classes of diagrams when we compute $\Hom_{C^{b}(k G\mmod)}(E,E[a])$ where 
\[
\xymatrixcolsep{0.6cm}
\xymatrixrowsep{0.5cm}
E:= \xymatrix{ 0 \ar[r] & P_{j} \ar[r] & P_{j+1} \ar[r] & P_{j+2} \ar[r] & \cdots \ar[r] & P_{m} \ar[r]& P_{m} \ar[r] & \cdots \ar[r] & P_{2} \ar[r]&  P_{1} \ar[r] & 0 }
\]
is the complex obtained by a truncated minimal projective resolution of $k$ as in the proof of \ref{Rickard3}.
\begin{lemma} \label{classdiag}
	Let $a$ be an integer. Assume that $|a|>1$. Then the non-zero maps between $E$  and $E[a]$ in $C^{b}(kG\mmod)$ are given by one of the following diagrams: 
	\begin{enumerate}
		\item  
		\[
		\xymatrix{ 
			\cdots \ar[r] &	P_{i-\epsilon} \ar[r] \ar[d]^{0} & P_{i} \ar[r] \ar[d]^{\neq 0}   & P_{i+\epsilon} \ar[d]^{0} \ar[r]& \cdots \\
			\cdots \ar[r]& P_{i+\epsilon} \ar[r] & P_{i} \ar[r] & P_{i-\epsilon} \ar[r]& \cdots}
		\]	
		for some $i \neq m$ and $\epsilon=\pm 1$, and 
		
		\item
		\[
		\xymatrix{ 
			\cdots \ar[r] \ar[d]^{0} & P_{i-\epsilon} \ar[r] \ar[d]^{f_1}  & P_{i} \ar[r] \ar[d]^{f_2}  & \cdots \ar[d]^{0} \\
			\cdots \ar[r] & P_{i} \ar[r] & P_{i -\epsilon} \ar[r] & \cdots 
		}
		\]	
		for some $i$  and $\epsilon= \pm 1$. 
	\end{enumerate}
\end{lemma}	
\begin{proof}
	Assume that $\ell \mid \Phi_m(q)$. Then the Brauer
	tree of the principal $\ell$-block $\Gl_n(q)$ is of the form
	\begin{center}
		\begin{pspicture}(13,0.5)
		\psset{linewidth=1pt}
		\cnode[fillstyle=solid,fillcolor=black](0,0){5pt}{A}
		\cnode(1.7,0){5pt}{B}
		\cnode(3.7,0){5pt}{C}
		\cnode(5.7,0){5pt}{D}
		\cnode(7.4,0){5pt}{E}
		\cnode(9.4,0){5pt}{F}
		\cnode(11.1,0){5pt}{G}
		\cnode(12.8,0){5pt}{H}
		\ncline[nodesep=0pt]{A}{B}\ncput*[npos=0.5]{$m$}
		\ncline[nodesep=0pt]{B}{C}\ncput*[npos=0.5]{$m-1$}
		\ncline[nodesep=0pt,linestyle=dashed]{C}{D}\naput[npos=0]{$\vphantom{\big)}$}
		\ncline[nodesep=0pt]{D}{E}\ncput*[npos=0.5]{$j$}
		\ncline[nodesep=0pt,linestyle=dashed]{E}{F}\naput[npos=0]{$\vphantom{\big)}$}
		\ncline[nodesep=0pt]{F}{G}\ncput*[npos=0.5]{$2$}
		\ncline[nodesep=0pt]{G}{H}\ncput*[npos=0.5]{$1$}
		\end{pspicture}
	\end{center}
	\vspace{0.5cm}
	Recall that $\Hom_{kG\mmod}(P_t,P_s)=0$ whenever $|s-t|>1$.
	The complex $E$ is given by 
	\[
	\xymatrixcolsep{0.3cm}
	\xymatrixrowsep{0.3cm}
	\xymatrix{ E = 0 \ar[r] & P_j \ar[r] & P_{j+1} \ar[r] & \cdots \ar[r] & P_{m-1} \ar[r] & P_{m} \ar[r] & P_m \ar[r] & P_{m-1} \ar[r] & \cdots \ar[r] & P_j \ar[r] & \cdots \ar[r] & P_2 \ar[r] & \ar[r] P_1 \ar[r] & 0 }.
	\]
	As we mentioned before the index $j$ depends on the characteristic $\ell$.
	\\
	
	Case $\mathbf{I}$: We first look at the left-hand side of $(P_m \longrightarrow P_m)$ in the complex $E$. We can write it as follows where  $i,i+1 < m$
	\[
	\xymatrix{ \cdots \ar[r] & P_{i-1} \ar[r] & P_i \ar[r] & P_{i+1} \ar[r] & P_{i+2} \ar[r] & \cdots \ar[r] & P_m \ar[r] & P_m \ar[r] & \cdots}.
	\]
	Let $f$ be a morphism between $E$ and $E[a]$ which we draw in the following diagram
	\[
	\xymatrix{E:  \cdots \ar[r] & P_{i-1} \ar[r] \ar[d]^{f_{i-1}} & P_i \ar[r] \ar[d]^{f_{i}} & P_{i+1} \ar[r] \ar[d]^{f_{i+1}} & P_{i+2} \ar[r] \ar[d]^{f_{i+2}} & \cdots \\
		E[a]: \cdots \ar[r] & P_{\alpha} \ar[r] & P_{\beta} \ar[r] & P_{\gamma} \ar[r] &  P_{\delta} \ar[r] & \cdots}
	\]
	Assume that $f_i \neq 0 $ and that $f_{i'}=0$ for all $i' < i$. We have  $\Hom_{kG\mmod}(P_i,P_{\beta}) \neq 0$ if and only if $\beta =i, i+1, i-1$ and it gives us $(\alpha,\beta,\gamma)$ as one of the following form
	\begin{center}
		$\{(i+1,i,i-1), (i-1,i,i+1), (i, i+1, i+2), (i+2,i+1,i), (i,i-1,i-2), (i-2,i-1,i)  \}$.
	\end{center}
	Since $|a| >1$ then only the following cases can actually occur 
	$$\{(i+1,i,i-1), (i+2,i+1,i), (i,i-1,i-2) \}.$$
	$\bullet$ Let $(\alpha,\beta,\gamma)=(i+1,i,i-1)$ then, we obtain the following diagram and only one non-zero map:
	\[
	\xymatrix{ \cdots \ar[r] & P_{i-1} \ar[r] \ar[d]^{0} & P_i \ar[r] \ar[d]^{\neq 0} & P_{i+1} \ar[r] \ar[d]^{0} & P_{i+2} \ar[r] \ar[d]^{0} & \cdots \\
		\cdots \ar[r] & P_{i+1} \ar[r] & P_{i} \ar[r] & P_{i-1} \ar[r] &  P_{i-2} \ar[r] & \cdots}
	\]
	It corresponds to the map (i) with $\epsilon = 1$.
	\\
	$\bullet$ If $(\alpha,\beta,\gamma)=(i+2,i+1,i)$ then we have  two non zero maps.
	\[
	\xymatrix{ \cdots \ar[r] & P_{i-1} \ar[r] \ar[d]^{0} & P_i \ar[r] \ar[d]^{\neq 0} & P_{i+1} \ar[r] \ar[d]^{\neq 0} & P_{i+2} \ar[r] \ar[d]^{0} & \cdots \\
		\cdots \ar[r] & P_{i+2} \ar[r] & P_{i+1} \ar[r] & P_{i} \ar[r] &  P_{i-1} \ar[r] & \cdots}
	\]
	It corresponds to the maps (ii) with $\epsilon = 1$.
	\\
	$\bullet$ If $(\alpha,\beta,\gamma)=(i,i-1,i-2)$ then similarly we obtain one non-zero map between the complexes.
	\[
	\xymatrix{ \cdots \ar[r] & P_{i-1} \ar[r] \ar[d]^{0} & P_i \ar[r] \ar[d]^{\neq 0} & P_{i+1} \ar[r] \ar[d]^{0} & P_{i+2} \ar[r] \ar[d]^{0} & \cdots \\
		\cdots \ar[r] & P_{i} \ar[r] & P_{i-1} \ar[r] & P_{i-2} \ar[r] &  P_{i-3} \ar[r] & \cdots}
	\]
	It corresponds to the maps (ii) with $\epsilon = 1$ with $f_1$ being zero.
	\smallskip

	Case $\mathbf{II}$: We consider the right hand side of  $(P_m \longrightarrow P_m)$ in the complex $E$. We present this complex as follows such that $i,i+1 <m$
	\[
	\xymatrix{\cdots \ar[r] & P_m \ar[r] & P_m \ar[r] & \cdots \ar[r] & P_{i+2} \ar[r] & P_{i+1} \ar[r] & P_{i} \ar[r] & P_{i-1} \ar[r] & \cdots }.
	\]
	Let $g$ be a morphism between $E$ and $E[a]$ which we draw in the following diagram
	\[
	\xymatrix{E:  \cdots \ar[r] & P_{i+2} \ar[r] \ar[d]^{g_{i+2}} & P_{i+1} \ar[r] \ar[d]^{g_{i+1}} & P_{i} \ar[r] \ar[d]^{g_{i}} & P_{i-1} \ar[r] \ar[d]^{g_{i-1}} & \cdots \\
		E[a]: \cdots \ar[r] & P_{\alpha} \ar[r] & P_{\beta} \ar[r] & P_{\gamma} \ar[r] &  P_{\delta} \ar[r] & \cdots}
	\]
	Assume that $g_i \neq 0 $ and that $g_{i'}=0$ for all $i' < i$.
	Similar to the case $\mathbf{I}$, $\Hom_{(kG\mmod)}(P_i,P_{\gamma})\neq 0$ if and only if $\beta =i, i+1, i-1$. Hence, we can obtain $(\beta,\gamma,\delta)$ as follows
	\begin{center}
		$\{ (i-1,i,i+1), (i,i+1,i+2), (i-2,i-1,i) \}$.
	\end{center}	
	The argument in this case is entirely similar to the previous case and we obtain the diagrams with one or two non-zero maps between two complexes. The only difference is that $\epsilon = -1$.
	\smallskip
	
	Case $\mathbf{III}$: We  are left with  maps of the form 
	\[
	\xymatrix{ \cdots \ar[d]^{0} \ar[r] & P_{m-1} \ar[r] \ar[d] & P_m \ar[r] \ar[d] & P_{m} \ar[r] \ar[d] & P_{m-1} \ar[r] \ar[d]  & \cdots  \ar[d]^{0}\\
		\cdots \ar[r] & P_{\alpha} \ar[r] & P_{\beta} \ar[r] & P_{\gamma} \ar[r] &  P_{\delta} \ar[r] & \cdots}
	\]
	Since $|a|>1$ then both $\beta$ and $\gamma$ are different from $m$.
	In this case, we have the following  possibilities for $(\alpha,\beta,\gamma,\delta)$.
	\begin{align*}
	\{ (m,m-1,m-2,m-3), (m-1,m-2,m-3,m-4), (m-2,m-3,m-4,m-5), \\ \phantom{\{ )} (m-3,m-2,m-1,m), (m-4,m-3,m-2,m-1),  (m-5,m-4,m-3,m-2) \}.
	\end{align*}	
	By symmetry we only look at the first three.
	
	\noindent $\bullet$ If $(\alpha,\beta,\gamma,\delta)=(m,m-1,m-2,m-3)$ then we obtain at most two non-zero maps
	\[
	\xymatrix{ \cdots \ar[d]^{0} \ar[r] & P_{m-1} \ar[r]^{\neq 0} \ar[d] & P_m \ar[r]^{\neq 0} \ar[d] & P_{m} \ar[r] \ar[d]^{0} & P_{m-1} \ar[r] \ar[d]^{0}  & \cdots  \ar[d]^{0}\\
		\cdots \ar[r] & P_{m} \ar[r] & P_{m-1} \ar[r] & P_{m-2} \ar[r] &  P_{m-3} \ar[r] & \cdots}
	\]
	which correspond to (ii).
	
	\noindent $\bullet$ Let $(\alpha,\beta,\gamma,\delta)=(m-1,m-2,m-3,m-4)$.  There is at most one non-zero map
	\[
	\xymatrix{ \cdots \ar[d]^{0} \ar[r] & P_{m-1} \ar[r]^{\neq 0} \ar[d] & P_m \ar[r] \ar[d]^{0} & P_{m} \ar[r] \ar[d]^{0} & P_{m-1} \ar[r] \ar[d]^{0}  & \cdots  \ar[d]^{0}\\
		\cdots \ar[r] & P_{m-1} \ar[r] & P_{m-2} \ar[r] & P_{m-3} \ar[r] &  P_{m-4} \ar[r] & \cdots}
	\]
	which corresponds to (i).
	\\
	$\bullet$ If $(\alpha,\beta,\gamma,\delta)=(m-2,m-3,m-4,m-5)$ there is at most one non-zero map 
	\[
	\xymatrix{ \cdots \ar[d]^{0} \ar[r] & P_{m-1} \ar[r]^{\neq 0} \ar[d] & P_m \ar[r] \ar[d]^{0} & P_{m} \ar[r] \ar[d]^{0} & P_{m-1} \ar[r] \ar[d]^{0}  & \cdots  \ar[d]^{0}\\
		\cdots \ar[r] & P_{m-2} \ar[r] & P_{m-3} \ar[r] & P_{m-4} \ar[r] &  P_{m-5} \ar[r] & \cdots}
	\]
	which corresponds to (ii).
\end{proof}	

\begin{thm} \label{partit}
	Assume $l \mid \Phi_m(q)$ where $n \geq m>n/2$ and $m>6$. Write the Brauer tree of the principal $\ell$-block of $ \Gl_n(q)$  as follows 
	\begin{center}
		\begin{pspicture}(13,0.5)
		\psset{linewidth=1pt}
		\cnode[fillstyle=solid,fillcolor=black](0,0){5pt}{A}
		\cnode(1.7,0){5pt}{B}
		\cnode(3.7,0){5pt}{C}
		\cnode(5.7,0){5pt}{D}
		\cnode(7.4,0){5pt}{E}
		\cnode(9.4,0){5pt}{F}
		\cnode(11.1,0){5pt}{G}
		\cnode(12.8,0){5pt}{H}
		\ncline[nodesep=0pt]{A}{B}\ncput*[npos=0.5]{$m$}
		\ncline[nodesep=0pt]{B}{C}\ncput*[npos=0.5]{$m-1$}
		\ncline[nodesep=0pt,linestyle=dashed]{C}{D}\naput[npos=0]{$\vphantom{\big)}$}
		\ncline[nodesep=0pt]{D}{E}\ncput*[npos=0.5]{$j$}
		\ncline[nodesep=0pt,linestyle=dashed]{E}{F}\naput[npos=0]{$\vphantom{\big)}$}
		\ncline[nodesep=0pt]{F}{G}\ncput*[npos=0.5]{$2$}
		\ncline[nodesep=0pt]{G}{H}\ncput*[npos=0.5]{$1$}
		\end{pspicture}
	\end{center}
	\vspace{0.5cm}
	Then for all $2 \leq j \leq n$ the complex
	\[
	\xymatrixcolsep{0.6cm}
	\xymatrixrowsep{0.5cm}
	\xymatrix{ E = 0 \ar[r] & P_{j} \ar[r] & P_{j+1} \ar[r] & P_{j+2} \ar[r] & \cdots \ar[r] & P_{m} \ar[r]& P_{m} \ar[r] & \cdots \ar[r] & P_{2} \ar[r]&  P_{1} \ar[r] & 0 }
	\]
	is a partial-tilting  complex.
\end{thm}
\begin{proof}
	We shall first show that the maps given in 	\ref{classdiag} are null-homotopic. We start with the second class (ii) in \ref{classdiag}. Without loss of generality we will assume that $\epsilon=1$ so we consider the following diagram 
	\[
	\xymatrix{ 
		\cdots \ar[r] \ar[d]^{0} & P_{i-1} \ar[r]^{d} \ar[d]^{f_1}  & P_{i} \ar[r] \ar[d]^{f_2}  & \cdots \ar[d]^{0} \\
		\cdots \ar[r] & P_{i} \ar[r]^{\delta} & P_{i -1} \ar[r] & \cdots 
	}
	\]		
	We are going to show that $f_1$ and $f_2$ are null-homotopic by constructing a single morphism $s: P_i \longrightarrow P_i$ such that $f_1=s \circ d$ and $f_2=\delta \circ s$. In other words, all the other $s_i$'s will be assumed to be zero.
	$$\xymatrix{ 
		\cdots \ar[r] \ar[d]^{0} & P_{i-1} \ar[r]^{d} \ar[ld]^0 \ar[d]^{f_1}  & P_{i} \ar[r] \ar[d]^{f_2} \ar@{-->}[ld]^s  & \cdots \ar[d]^{0} \ar[ld]^0 \\
		\cdots \ar[r] & P_{i} \ar[r]^{\delta} & P_{i -1} \ar[r] & \cdots 
	}
	$$
	Since $\Hom_{kG\mmod}(P_{i-1},P_{i})$ is one-dimensional, it is generated by the map $d$ therefore, $f_1= \lambda d$ for some scalar $\lambda \in k$. Similarly, since $\Hom_{kG\mmod}(P_{i},P_{i-1})$ is generated by $\delta $ then $f_2=\mu \delta $ for some scalar $\mu$. 
	We note that $f_1$ and $f_2$ are parts of a morphisms of complexes so $\delta \circ f_1=f_2 \circ d$ and it gives us $\delta \circ (\lambda d)=(\mu \delta) \circ d$. Since $\delta \circ d$ is non-zero endomorphism of $P_{i-1}$ (sending the top to the socle) it follows that $\lambda=\mu$. Therefore $s=\lambda \text{Id}$ gives us the expected homotopy.
	\smallskip

	For the class (i), we assume again without loss of generality that $\epsilon=1$. We shall construct the homotopy using a morphism $s:P_{i+1} \longrightarrow P_i$
	\[
	\xymatrix{ 
		P_{i-1} \ar[r] \ar[d]^{0} & P_{i} \ar[r]^{d} \ar[ld]^0 \ar[d]^{f}   & P_{i+1} \ar[d]^{0} \ar@{-->}[dl]^{s}  \\
		P_{i+1} \ar[r] & P_{i} \ar[r]^{\delta} & P_{i-1} 
	}
	\]	
	satisfying $f = s\circ d$. 
	To construct $s$ let us first note $f \in \End_{kG\mmod}(P_{i})$ can not be an isomorphism. Since $i \geq j > 1$ then $\delta \neq 0$. On the other hand we have $\delta \circ f =0$. 
	Therefore $f$ lies in the radical of the algebra $\End_{kG\mmod}(P_{i})$. Since it is a two dimensional algebra, this radical has dimension one. On the other hand, any non-zero map $s:P_{i+1} \longrightarrow P_i $ satisfies $s \circ d  \neq 0$ and is not an isomorphism. More precisely the composition $s \circ d$  will send the $\mathrm{top}$ of the module $P_{i}$ to its $\mathrm{socle}$, as we can see in the following diagram
	\begin{center}
		$\begin{pmatrix} 
		& \textcolor{red}{S_{i} } &        \\
		\textcolor{red}{S_{i+1}} &          &   S_{i-1}  \\
		&   S_{i}  &       \\
		\end{pmatrix}  \longrightarrow^{\hspace{-0.4cm}d}  \hspace{0.1cm}
		\begin{pmatrix} 
		& \textcolor{green}{S_{i+1}} &        \\
		S_{i+2} &          &   \textcolor{yellow}{S_{i} }  \\
		& \textcolor{red} {S_{i+1} }  &       \\
		\end{pmatrix} \longrightarrow ^{\hspace{-0.45cm}s}
		\begin{pmatrix} 
		& S_{i} &        \\
		\textcolor{green}{S_{i+1} } &          &   S_{i-1}  \\
		&   \textcolor{green}{S_{i} }  &       \\
		\end{pmatrix}.$
	\end{center}
	Consequently for any non-zero map $s:P_{i+1} \longrightarrow P_i$ the maps $s \circ d $ and $f$ differ by a scalar. Hence there exists $s$ such that $ s \circ d =f$. 
	
	\smallskip
	
	Now we will assume that $a= 1$. Let $f:E \longrightarrow E[1]$ be a morphism of complexes which we will write as 
	\[
	\xymatrixcolsep{0.5cm}
	\xymatrixrowsep{0.5cm}
	\xymatrix{ 
		0  \ar[r] \ar[d]^{0}  &	P_j \ar[r]^{d_j} \ar[d]^{f_j}   & \cdots \ar[r]  & P_{m-2} \ar[r]^{d_{m-2}} \ar[d] & P_{m-1} \ar[r]^{d_{m-1}} \ar[d]^{f_{m-1}} & P_{m} \ar[r]^{d_m} \ar[d]^{f_m} & P_{m} \ar[r]^{d'_{m}} \ar[d]^{f^{\prime}_{m}}  &  P_{m-1} \ar[r]^{d'_{m-2}} \ar[d]^{f^{\prime}_{m-1}}    &\cdots \ar[r]  & P_{1} \ar[r] \ar[d]^{0}   & 0 \\
		P_{j} \ar[r]_{d_j}  &	P_{j+1} \ar[r]_{d_{j+1}}  & \cdots \ar[r] &  P_{m} \ar[r] &P_m \ar[r]_{d_m}  & P_{m} \ar[r]_{d'_{m-1}}  & P_{m-1} \ar[r]_{d'_{m-2}}   &  P_{m-2}  \ar[r] & \cdots \ar[r]   & 0 \ar[r] & 0 
	}
	\]
	Since $\Hom_{kG\mmod}(P_i,P_{i+1})$ is one-dimensional then for all $i<m$ there exists $\lambda_i \in k$ such that $f_i=\lambda_id_i$. Similarly, for all $i\leq m$ there exists $\lambda_i' \in k$ such that $f'_i=\lambda'_i d'_i$.
	One can also relate $f_m$ to $d_m$ as follows; $d_m$ generates the Jacobson radical of $\End_{kG\mmod}(P_m)$. Let $r+1= \dim \End_{kG\mmod}(P_m)$. Since $d_m$ is nilpotent of order $r+1$, then we claim that $A:=\{d_m,d^2_m,...,d^{r}_m\}$ is a basis for the algebra $J(\End_{kG\mmod}(P_m))$. Indeed, since the radical has dimension $r$, it is enough to prove that $A$ is a linearly independent set. 
	If
	\begin{center}
		$0 = a_1d_m+a_2 d^2_m+\cdots +a_r d^r_m$,
	\end{center}
	we can multiply the two sides of the equality by $d_m^{r-1}$. It gives us
	\begin{center}
		$ 0 = a_1 d^{r}_m + a_2 d^{r+1}_m+ \cdots + a_r d^{2r-1}_m = a_1 d^{r}_m$.
	\end{center}	
	It follows that $a_1=0$. By successively multiplying by $d^{r-2}_m$, $d^{r-3}_m$,\ldots \ we obtain $a_2=0$, $a_3=0$,\ldots. We deduce that $A$ is a linearly independent set.  
	Therefore, one can write $f_m= \sum_{i>0}^{r}\alpha_i d_m^{i}$.
	Let us set $t:=\min\{i \mid \alpha_i \neq 0\}$.
	Since $f_m$ is not an isomorphism then $t \geq 1$. Then we can rewrite $f_m$ as:
	$$
	f_m= \displaystyle  d_m^{t}\left(\sum_{i \geq t}\alpha_i d_m^{i-t}\right).
	$$
	We set $\gamma:=\sum_{i\geq t}\alpha_i d_m^{i-t} \in \alpha_t \mathrm{Id} + kA$. Since  $\alpha_t \neq 0$ then $\gamma$ is an isomorphism of $P_m$. Note that $d_m$ and $\gamma$ commute. 
	\smallskip
	
	Now we construct the homotopy. Let us consider the following diagram
	\[
	\xymatrixcolsep{0.6cm}
	\xymatrixrowsep{0.5cm}
	\xymatrix{ 
		0  \ar[r] \ar[d]^{0}  &	P_j \ar[r]^{d_j} \ar[d]^{f_j} \ar@{-->}[dl]^{s_j}   & \cdots \ar[r]  & P_{m-2} \ar[r]^{d_{m-2}} \ar[d] & P_{m-1} \ar[r]^{d_{m-1}} \ar[d]^{f_{m-1}} \ar@{-->}[dl]^{s_{m-1}} & P_{m} \ar[r]^{d_m} \ar[d]^{f_m} \ar@{-->}[dl]^{s_{m}}  & P_{m} \ar[r]^{d'_{m}} \ar[d]^{f^{\prime}_{m}} \ar@{-->}[dl]^{s'_{m}}  &  P_{m-1} \ar[r]^{d'_{m-2}} \ar[d]^{f^{\prime}_{m-1}}  \ar@{-->}[dl]^{s'_{m-1}}   &\cdots \ar[r]  & P_{1} \ar[r] \ar[d]^{0}   & 0 \\
		P_{j} \ar[r]_{d_j}  &	P_{j+1} \ar[r]_{d_{j+1}}  & \cdots \ar[r] &  P_{m-1} \ar[r]^{d_{m-1}} &P_m \ar[r]_{d_m}  & P_{m} \ar[r]_{d'_{m}}  & P_{m-1} \ar[r]_{d'_{m-2}}   &  P_{m-2}  \ar[r] & \cdots \ar[r]   & 0 \ar[r] & 0 
	}
	\]
	We already have $f_m=  d^{t}_m \circ \gamma = (d^{t-1}_m \circ \gamma) \circ d_m$ since $t \geq 1$. Now if we  set
	\begin{center}
		$s_m:= d_m^{t-1} \circ \gamma$ \quad and \quad $s_m' := 0$
	\end{center}
	then $f_m=d_m \circ s_m + s_m' \circ d_m$. Since $\Hom_{kG\mmod}(P_{m-1},P_m)$ is one-dimensional, there exists a scalar $\lambda_m$ such that $s_m \circ d_{m-1}=\lambda_m d_{m-1}$. Note that $\lambda_m = 0$ whenever $t > 1$. 
	
	\smallskip 
	
	Now we can define the maps $s_i$ and $s_i'$ for all $i < m$ by 
	$$  s_i:= \left(\sum_{a=i}^{m}(-1)^{a+i} \lambda_a\right) \mathrm{Id}_{P_i} \quad 
	\text{and} 
	\quad s'_i:=\left(\sum^{m}_{a=i+1}(-1)^{a+i-1}\lambda'_a\right)\mathrm{Id}_{P_i}.$$
	Then for all $i < m-1$ we have
	$$d_i \circ s_i + s_{i+1} \circ d_i =  \left(\sum_{a=i}^{m}(-1)^{a+i} \lambda_a\right) d_i + \left(\sum_{a=i+1}^{m}(-1)^{a+i+1} \lambda_a\right) d_i = \lambda_i d_i = f_i$$
	and similarly $d_i'\circ s'_i +s'_{i-1} \circ d'_i = \lambda_{i}' d_i' = f_i'$ for all $i \leq m$. Finally, by definition of $\lambda_m$ we have
	$$ \begin{aligned} s_m \circ d_{m-1}+d_{m-1}\circ s_{m-1} & \, = \lambda_{m} d_{m-1} + d_{m-1}\circ s_{m-1}  \\& \,= \lambda_{m} d_{m-1} +  (\lambda_{m-1} - \lambda_{m})d_{m-1} \\& \, = \lambda_{m-1}d_{m-1} \\ & \, = f_{m-1}. \end{aligned}$$
	Therefore $f$ is null-homotopic. A similar argument applies for morphisms $E \longrightarrow E[-1]$.
\end{proof}

\begin{center}
	\textbf{ACKNOWLEDGEMENTS}
\end{center}	
I would like to thank my Ph.D advisor, Olivier Dudas, for all his guidance and support. I would also like to thank the referee for giving very useful remarks. This project has received funding from the European Union’s Horizon 2020 research and innovation program under the Marie Sklodowska-Curie grant agreement No 665850.

\newpage
\bibliographystyle{abbrv} 
\bibliography{thebib1}

\end{document}